\documentclass[12pt]{amsart}
\usepackage{txfonts,breakurl}
\usepackage{amscd,amssymb,mathabx}
\usepackage[all]{xy}
\usepackage{graphicx,calrsfs}
\usepackage[colorlinks,plainpages,backref,urlcolor=blue]{hyperref}
\usepackage{mdwlist}
\usepackage{enumitem}

\newtheorem{theorem}{Theorem}[section]
\newtheorem{corollary}[theorem]{Corollary}
\newtheorem{lemma}[theorem]{Lemma}
\newtheorem{prop}[theorem]{Proposition}
\newtheorem{claim}[theorem]{Claim}

\theoremstyle{definition}

\newtheorem{example}[theorem]{Example}
\newtheorem{remark}[theorem]{Remark}
\newtheorem{construction}[theorem]{Construction}

\newtheorem{question}[theorem]{Question}

\newtheorem*{ack}{Acknowledgments}

\newcommand{\Z}{\mathbb{Z}}
\newcommand{\Q}{\mathbb{Q}}
\newcommand{\R}{\mathbb{R}}
\newcommand{\C}{\mathbb{C}}

\renewcommand{\L}{\mathbb{L}}

\newcommand{\CP}{\mathbb{CP}}
\renewcommand{\k}{\Bbbk}
\newcommand{\K}{\mathbb{K}}

\DeclareMathAlphabet{\pazocal}{OMS}{zplm}{m}{n}

\newcommand{\XX}{{\pazocal X}}

\newcommand{\RR}{{\mathcal R}}
\newcommand{\VV}{{\mathcal V}}

\newcommand{\F}{{\mathcal{F}}}

\newcommand{\M}{{\mathcal{M}}}

\newcommand{\g}{{\mathfrak{g}}}
\newcommand{\h}{{\mathfrak{h}}}
\newcommand{\gl}{{\mathfrak{gl}}}
\renewcommand{\sl}{{\mathfrak{sl}}}
\newcommand{\m}{{\mathfrak{m}}}

\DeclareMathOperator{\rank}{rank}
\DeclareMathOperator{\gr}{gr}
\DeclareMathOperator{\im}{im}

\DeclareMathOperator{\id}{id}

\DeclareMathOperator{\GL}{GL}
\DeclareMathOperator{\SL}{SL}
\DeclareMathOperator{\SU}{SU}
\DeclareMathOperator{\Sp}{Sp}
\DeclareMathOperator{\Hom}{{Hom}}
\DeclareMathOperator{\spn}{span}

\DeclareMathOperator{\proj}{pr}

\DeclareMathOperator{\ad}{ad}

\DeclareMathOperator{\Heis}{\mathcal{H}}
\DeclareMathOperator{\PD}{PD}
\DeclareMathOperator{\ideal}{ideal}
\DeclareMathOperator{\wt}{wt}
\DeclareMathOperator{\orb}{orb}

\newcommand{\surj}{\twoheadrightarrow}
\newcommand{\inj}{\hookrightarrow}

\newcommand{\isom}{\xrightarrow{\,\simeq\,}}

\def\dot{\mathchar"013A}  
\newcommand{\hdot}{{\raise2pt\hbox to0.35em{\Huge $\dot$}}} 
\newcommand{\bwedge}{\mbox{\small $\bigwedge$}}

\newcommand{\cdga}{\textsc{cdga }}
\newcommand{\dga}{\textsc{cdga}}

\begin{document}

\title[Compact Lie group actions and finite models]{%
The topology of compact Lie group actions through the lens of finite models}

\author[Stefan Papadima]{Stefan Papadima$^1$}
\address{Simion Stoilow Institute of Mathematics, 
P.O. Box 1-764,
RO-014700 Bucharest, Romania}
\email{\href{mailto:Stefan.Papadima@imar.ro}{Stefan.Papadima@imar.ro}}
\thanks{$^1$Partially supported by the Romanian Ministry of 
National Education, CNCS-UEFISCDI, grant PNII-ID-PCE-2012-4-0156}

\author[Alexander~I.~Suciu]{Alexander~I.~Suciu$^2$}
\address{Department of Mathematics,
Northeastern University,
Boston, MA 02115, USA}
\email{\href{mailto:a.suciu@neu.edu}{a.suciu@neu.edu}}
\urladdr{\href{http://www.northeastern.edu/suciu/}%
{www.northeastern.edu/suciu/}}
\thanks{$^2$Partially supported by 
NSA grant H98230-13-1-0225 and the Simons Foundation 
collaborative grant 354156}

\subjclass[2010]{Primary
55P62,  
Secondary
14F35,  
17B70,   
53C25,  
55R40,  
57S25. 
}

\keywords{Differential graded algebra, formality, holonomy Lie algebra, 
Malcev completion, representation variety, cohomology jump loci, compact 
Lie group action, characteristic classes, regular sequence, Poincar\'{e} duality, 
Sasakian geometry, isolated surface singularity.}

\begin{abstract}
Given a compact, connected Lie group $K$, we use principal $K$-bundles 
to construct manifolds with prescribed finite-dimensional algebraic models.
Conversely, let $M$ be a compact, connected, smooth manifold which 
supports an almost free $K$-action.  Under a partial formality assumption 
on the orbit space and a regularity assumption on the characteristic classes 
of the action, we describe an algebraic model for $M$ with commensurate 
finiteness and partial formality properties.  The existence of such a model 
has various implications on the structure of the cohomology jump loci of 
$M$ and of the representation  varieties of $\pi_1(M)$.  As an application, 
we show that compact Sasakian manifolds of dimension $2n+1$ are 
$(n-1)$-formal, and that their fundamental groups are filtered-formal.  
Further applications to the study of weighted-homogeneous isolated 
surface singularities are also given.
\end{abstract}

\maketitle
\setcounter{tocdepth}{1}
\tableofcontents

\section{Introduction and statement of results}
\label{sect:intro}

\subsection{Algebraic models for spaces}
\label{subsec:intro1}

The rational homotopy type of a nilpotent CW-space of finite type can be reconstructed 
from algebraic models associated to it.  If the space in question is a 
smooth manifold $M$, the standard model is the de Rham algebra 
$\Omega^{\hdot}_{\rm dR}(M)$ of smooth forms on the manifold, 
endowed with a wedge product and differential satisfying the 
graded Leibniz rule.  Sullivan \cite{Su77} associated 
to any space $X$ a commutative, differential graded $\C$-algebra (for short, 
a $\dga$), denoted by $\Omega^{\hdot}(X)$, which serves as the reference 
algebraic model for the space.  In particular, $H^{\hdot}(\Omega(X))=H^{\hdot}(X,\C)$.

We are interested here in algebraic models of a connected CW-complex 
$X$ which satisfy certain finiteness 
and formality conditions.  As shown in \cite{DP-ccm}, models with finiteness 
properties contain valuable topological information related to the structure 
around the origin of both representation varieties of fundamental groups 
and the loci where the corresponding twisted homology of spaces (in degrees 
up to a fixed $q\ge 1$) jumps. This requires the finiteness of the $q$-skeleton of $X$, but {\em no}  
nilpotence condition. When finiteness of models comes from stronger formality 
properties of $X$, additional nice features hold.   

More precisely, we say that $A^{\hdot}$ is {\em $q$-finite}\/ if $A$ is connected 
(i.e., $A^0=\C$) and $A^i$ is finite-dimensional, for each $i\le q$.  
Likewise, we say that $(A^{\hdot},d)$ is 
a {\em $q$-model}\/ for $X$ if $A$ has the same $q$-type as $\Omega^{\hdot}(X)$, 
i.e., there is a zig-zag of morphisms connecting these two $\dga$s, with each such 
morphism inducing isomorphisms in homology up to degree $q$ and a 
monomorphism in degree $q+1$. Furthermore, we say that 
$(A,d)$ is {\em $q$-formal}\/ if it can be connected to $(H^{\hdot}(A),d=0)$ by such 
a zig-zag of $\dga$ morphisms. Finally, we say $X$ is $q$-formal if $\Omega^{\hdot}(X)$ 
has this property. All these notions have obvious analogues when $q=\infty$, 
in which case we drop the prefix $q$. 

Let $K$ be a compact, connected real Lie group. If a connected space $X$ 
admits a finite model, it is known that any principal $K$-bundle over $X$ 
has the same property. Our first result (proved in Theorem \ref{thm:krealiz} 
and Corollary \ref{cor:newpd}) uses this fact to provide a systematic way 
of constructing a rich variety of new, interesting spaces admitting finite models.
To state this result, we identify the cohomology algebra $H^{\hdot}(K,\Q)$ 
with an exterior algebra on odd-degree generators, $\bwedge (t_1,\dots ,t_{r})$, 
where $r=\rank K$.

\begin{theorem}
\label{thm:krealiz-intro}
Let $X$ be a connected, finite CW-space. If $X$ has a finite, rational model $B$, 
then any Hirsch extension $A=B\otimes_{\tau} \bwedge (t_1,\dots ,t_{r})$
can be realized as a finite, rational model of some principal $K$-bundle $Y$ over $X$. 
Moreover, if the $\dga$ $B$ satisfies Poincar\'{e} duality in dimension $n$, then $A$ 
has the same property in dimension $n+ \dim K$. 
\end{theorem}

As is well-known, the formality property of a space $X$ is not necessarily 
inherited by a principal $K$-bundle $Y$ over $X$. Our next result 
gives a useful algebraic criterion on the characteristic classes
of the bundle which implies the $q$-formality of $Y$.

Let $H^{\hdot}$ be a connected, commutative 
graded algebra, and let  $\{e_\alpha\}$ be a sequence of homogeneous elements 
of degree $n_\alpha>0$.   We say that this sequence is {\em $q$-regular}\/ if for 
each $\alpha$, the class of $e_{\alpha}$ in the quotient algebra $\overline{H}_{\alpha}=
H/\sum_{\beta <\alpha} e_{\beta} H$ has trivial annihilator up to degree $q-n_{\alpha}+2$.

\begin{theorem}
\label{thm:intro1}
Suppose $e_1,\dots ,e_r$ is an even-degree, $q$-regular sequence in $H^{\hdot}$.  
Then the Hirsch extension $A=H \otimes_{\tau} \bwedge (t_1,\dots ,t_r)$ with 
$d=0$ on $H$  and $dt_{\alpha}=\tau (t_{\alpha})=e_{\alpha}$ has 
the same $q$-type as $(H/\sum_{\alpha} e_{\alpha} H, d=0)$.
In particular, $A$ is $q$-formal.
\end{theorem}

For $q=\infty$, this is a classical property of Koszul complexes from commutative 
algebra, see \cite[Proposition 3.6]{H87}.  Work of Borel \cite{Borel} and 
Chevalley \cite{Chev} supplies examples of fully regular sequences in 
graded polynomial rings which arise in the context of compact connected 
Lie groups and finite reflection groups. 

The above notion of partial regularity seems well adapted to the 
study of cohomology rings of finite-dimensional CW-complexes. 
The Hard Lefschetz Theorem provides natural 
examples of one-element partially regular sequences in the cohomology rings of 
compact K\"{a}hler manifolds, known to be formal by work of Deligne et al.~\cite{DGMS}. 
Theorem  \ref{thm:intro1} may be applied to analyze the partial formality of 
arbitrary homogeneous spaces of compact, connected Lie groups. 
As shown in Example \ref{ex=homnonf}, the resulting estimates 
are sharp in certain cases. 

\subsection{Almost free $K$-actions}
\label{subsec:intro2}

We now switch to a topological context which is broader than principal bundles. 
Consider a compact, connected, smooth manifold $M$ (without boundary) 
on which a compact, connected Lie group $K$ acts smoothly. 
The $K$-action on $M$ is said to be {\em almost free}\/ if all its 
isotropy groups are finite.  Let $K\to EK\times M \to M_K$ 
be the Borel construction on $M$, and let $\tau\colon H^{\hdot}(K,\C) 
\to H^{\hdot +1}(M_K,\C)$ be the transgression
in the Serre spectral sequence of this bundle.  

Let $N=M/K$ be the orbit space.  
The projection map $\proj \colon M_K \to N$ induces an isomorphism 
 $\proj^* \colon H^{\hdot}(N,\C) \to H^{\hdot}(M_K,\C)$.  
We identify the cohomology algebra 
$H^{\hdot}(K,\C)$ with the exterior algebra on odd-degree generators 
$t_1,\dots ,t_{r}$, as before, and we let 
$e_{\alpha}=(\proj^*)^{-1}(\tau(t_{\alpha})) \in H^{m_{\alpha} +1} (N,\C)$ 
be the corresponding characteristic classes.

It is known that a model for $M$   
can be built from the Sullivan model $\Omega^{\hdot}(N)$ and the 
above characteristic classes, via a suitable Hirsch extension with 
$\bwedge P_K$, where $P_K=\spn\{t_1,\dots ,t_{r}\}$. 
Our next result uses this fact, together with Theorem \ref{thm:intro1}, to 
construct a $q$-model for $M$ which is both $q$-finite and 
$q$-formal, under suitable partial formality conditions on the 
orbit space $N$ and regularity conditions on the sequence 
$\{ e_{\alpha}\}$, and for $q$ in a certain range. 

\begin{theorem}
\label{thm:intro2}
Suppose the $K$-action on $M$ is almost free, the orbit space $N=M/K$ 
is $k$-formal, for some $k> \max \{ m_{\alpha} \}$, and the characteristic 
classes $e_1,\dots ,e_r$ form a $q$-regular sequence in the graded ring 
$H^{\hdot}=H^{\hdot}(N,\C)$, for some $q\le k$.  
Then the $\cdga$
$\big(H^{\hdot} \big/\sum_{\alpha} e_{\alpha} H^{\hdot}\big., d=0\big)$ 
is a finite $q$-model for $M$.  In particular, $M$ is $q$-formal.
\end{theorem}

In Example \ref{ex=notsharp}, we construct for an arbitrary $K$ a family of almost free
$K$-manifolds for which we are able to describe a lower bound for partial formality using the 
above result. In Example \ref{ex:hopf-nonformal}, this type of estimate becomes optimal.

The $1$-formality property of a connected CW-complex $X$ with 
finite $1$-skeleton depends only on its fundamental group, $\pi=\pi_1(X)$.  
Furthermore, the $1$-formality of $\pi$ is equivalent to 
the Malcev Lie algebra $\m(\pi)$ constructed by Quillen in \cite{Qu} 
being filtered-isomorphic to the degree 
completion of a quadratic, finitely generated Lie algebra $L$.  If $L$ is 
merely assumed to have homogeneous relations, then the group $\pi$ 
is said to be {\em filtered formal}. 

\begin{theorem}
\label{thm:intro3}
Let $\pi=\pi_1(M)$ be the fundamental group of a compact, connected 
manifold $M$ on which a compact, connected Lie group $K$ acts almost 
freely, with $2$-formal orbit space. Then:
\begin{enumerate}
\item \label{z1}
The group $\pi$ is filtered-formal. 
More precisely, the Malcev Lie algebra $\m(\pi)$ is isomorphic 
to the degree completion of $\L/\mathfrak{r}$, where $\L$ is 
the free Lie algebra on $H_1(\pi, \C)$ and $\mathfrak{r}$ 
is a homogeneous ideal generated in degrees $2$ and $3$. 

\item \label{z2}
For every complex linear algebraic group $G$, the germ at 
the trivial representation $1$ of the representation variety 
$\Hom_{\gr}(\pi,G)$ is defined by quadrics and cubics only.
\end{enumerate}
\end{theorem}

This result has the same flavor as the weight obstructions of Morgan \cite{Mo} 
and Kapovich--Millson \cite{KM}, which are specific to fundamental groups of 
quasi-projective manifolds and germs of their representation varieties, respectively. 
On the other hand, our theorem may be applied to arbitrary principal $K$-bundles 
over formal, compact manifolds. 

\subsection{Cohomology jump loci}
\label{subsec:intro4}
 
The {\em characteristic varieties}\/  of a space $X$ 
with respect to a rational representation 
$\iota \colon G\to \GL(V)$  are the sets $\VV^i_s(X,\iota)$ 
consisting of those representations $\rho\colon \pi\to G$ 
for which the twisted cohomology group $H^i(X , V_{\iota\circ \rho})$ 
has $\C$-dimension at least $s$.  In degree $i=1$, these 
varieties depend only on the group $\pi=\pi_1(X)$, and so 
we may denote them as $\VV^1_s(\pi,\iota)$.  
In the rank $1$ case, i.e., when $G=\C^*$ and $\iota=\id_{\C^*}$, 
we simply write the characteristic varieties of $X$ as $\VV^i_s(X)$.

When $M$ is a closed, orientable manifold of dimension $n$, 
it is known that Poincar\'{e} duality imposes subtle restrictions 
on these global jump loci.  Now suppose that $X$ is a finite 
CW-complex admitting a finite model $A$, 
which is an $n$-dimensional Poincar\'{e} 
duality $\dga$. Then, as we show in Theorem \ref{thm:invjump}, 
there is an analytic involution of the germ at the origin, 
$\Hom (\pi_1(X), G)_{(1)}$, which identifies 
$\VV^{i}_s(X, \iota)_{(1)}$ with $\VV^{n-i}_s(X, \iota)_{(1)}$, 
for all $i,s$, where $\iota$ is the identity representation of $\GL (V)$.
Furthermore, in the rank $1$ case, this involution is induced by 
the involution $\rho \mapsto \rho^{-1}$ of the character group 
$\Hom(\pi_1(X), \C^*)$. 

Let again $M$ be a smooth, closed manifold supporting 
an almost free $K$-action. 
The projection $p\colon M\to M/K$ induces a natural epimorphism 
$p_{\sharp}\colon \pi_1 (M) \surj \pi_1^{\orb} (M/K)$ between orbifold 
fundamental groups.  Our next result establishes a tight connection 
between the germs at the origin of the rank $1$ characteristic varieties 
and the $\SL_2(\C)$ representation varieties of  $\pi_1 (M)$ and 
$\pi_1^{\orb} (M/K)$. 

\begin{theorem}
\label{thm:intro4}
Suppose that the transgression 
$P^{\hdot}_K \to H^{\hdot +1}(M_K) \cong H^{\hdot +1}(M/K)$ is injective in degree $1$.
Then the following hold.
\begin{enumerate}
\item \label{im1}
If the orbit space $N=M/K$ has a $2$-finite $2$-model,
then the epimorphism $p_{\sharp}$ induces a local analytic isomorphism between 
$\VV^1_s (\pi_1 (M))_{(1)}$ and $\VV^1_s (\pi_1^{\orb} (N))_{(1)}$, 
for all $s$.
\item \label{im2}
If $N$ is $2$-formal, then $p_{\sharp}$ induces an analytic isomorphism between the 
germs at $1$ of $\Hom (\pi_1 (M), G)$ and $\Hom (\pi_1^{\orb} (N), G)$, for 
$G=\SL_2 (\C)$ or a Borel subgroup.
\end{enumerate}
\end{theorem}

For instance, let $K$ be a compact, connected Lie group, and 
let $N$ be a compact, formal manifold with $b_2(N)\ge \dim P_K^1$. 
Applying Theorem \ref{thm:krealiz-intro}, we may use these data to 
construct interesting principal $K$-bundles $M\to N$ which fit into the purview of 
Theorem \ref{thm:intro4}; see Example \ref{ex=tauinj} for more details.

In the case when $K=S^1$, suppose $M$ is a closed, orientable Seifert fibered 
$3$-manifold with non-zero Euler class, and let $p\colon M \to M/S^1 =\Sigma_g$
be the projection map onto a Riemann surface of genus $g$.
Then the above theorem insures that 
$\VV^1_s (M)_{(1)}\cong\VV^1_s (\Sigma_g)_{(1)}$, 
for all $s$, and $\Hom (\pi_1 (M), \SL_2 (\C))_{(1)}\cong 
\Hom (\pi_1 (\Sigma_g), \SL_2 (\C))_{(1)}$, as analytic germs 
at the trivial representation $1$; see Corollary \ref{cor=seifert}.

\subsection{Sasakian manifolds}
\label{subsec:intro3}
In many ways, Sasakian geometry is an odd-dimensional analogue of K\"{a}hler 
geometry.  More explicitly, every compact Sasakian manifold $M$ 
admits an almost-free circle action with orbit space $N=M/S^1$ 
a K\"{a}hler orbifold.  Furthermore, the Euler class of the action 
coincides with the K\"{a}hler class of the base, $h\in H^2(N,\Q)$, 
and this class satisfies the Hard Lefschetz property.

As shown by Tievsky in \cite{Ti}, every Sasakian manifold $M$ as above 
has a rationally defined, finite model of the form $(H^{\hdot}(N)\otimes  \bwedge(t), d)$, 
where the differential $d$ vanishes on $H^{\hdot}(N)$ and sends $t$ to $h$.  
Using this model and Theorem \ref{thm:intro1}, we prove in \S\ref{sect:sasaki} 
the following theorem, which is the Sasakian analog of the main result from 
\cite{DGMS}, namely, the formality of compact K\"{a}hler manifolds.

\begin{theorem}
\label{thm:intro5}
Let $M$ be a compact Sasakian manifold of dimension $2n+1$.  
Then $M$ is $(n-1)$-formal. 
\end{theorem}

We make this statement more precise in Theorem \ref{thm:highf-bis},  
by providing an explicit, finite $(n-1)$-model with zero differential 
for $M$ over an {\em arbitrary}\/ field of characteristic $0$. 
This result is optimal; indeed,  for each $n\ge 1$, the $(2n+1)$-dimensional 
Heisenberg compact nilmanifold $\Heis_n$ is a Sasakian manifold, 
yet it is not $n$-formal. 
Theorem \ref{thm:intro5} strengthens a statement of Kasuya \cite{Ka14}, 
who claimed that, for $n\ge 2$, a Sasakian manifold as 
above is $1$-formal. The proof of that claim, though, has 
a gap, which we avoid by giving a proof based on a very 
different approach. The formality of compact Sasakian manifolds 
is also analyzed by Biswas et al.~in \cite{BFMT}, using Massey products. 
Furthermore, Mu\~{n}oz and Tralle show in \cite{MT} that a 
compact, simply-connected, $7$-dimensional Sasakian manifold 
is formal if and only if all triple Massey products vanish.

A group $\pi$ is said to be a {\em Sasakian group}\/ if it 
can be realized as the fundamental group of a compact 
Sasakian manifold.  A major open problem in the field 
(cf.~\cite{BG, Chen}) is to  characterize this class of groups.  
Applying the previous theorems, we obtain new, rather 
stringent restrictions on a finitely presented group $\pi$ being Sasakian.  

\begin{theorem}
\label{thm:intro6}
Let $\pi=\pi_1(M^{2n+1})$ be the fundamental group of a compact 
Sasakian manifold of dimension $2n+1$.  Then:
\begin{enumerate}
\item \label{sas1}
The group $\pi$ is filtered-formal, and in fact $1$-formal if $n>1$.  
\item \label{sas2}
All irreducible components of the characteristic varieties $\VV^1_s(\pi)$ 
passing through the identity are algebraic subtori of the 
character group $\Hom(\pi,\C^*)$.
\item \label{sas3}
If $G$ is a complex linear algebraic group, then 
the germ at the origin of $\Hom(\pi, G)$ is defined by 
quadrics and cubics only, and in fact by quadrics only if $n>1$.
\end{enumerate}
\end{theorem}

Furthermore, if $G=\SL_2(\C)$ or a Borel subgroup thereof, then the germs 
at $1$ of $\Hom(\pi_1(M), G)$ and $\VV^1_s(\pi_1(M))$, for all $s$, depend 
(in an explicit way) only on the graded ring $H^{\hdot}(M/S^1, \C)$, while the 
Tievsky model also depends in an essential manner on the K\"{a}hler class 
$h\in H^2(M/S^1,\Q)$; see Corollary \ref{cor=nokclass}.

\subsection{Quasi-projective manifolds}
\label{subsec:intro5}

The infinitesimal analogue of the $G$-represen\-tation variety 
$\Hom(\pi,G)$ around the origin $1$ 
is the set of $\g$-valued flat connections on a $\dga$ $A$, where 
$\g$ is the Lie algebra of the Lie group $G$.  By definition,
this is the set $\F(A,\g)$ of elements in $A^1\otimes \g$ which 
satisfy the Maurer--Cartan equation, 
$d\omega +\frac{1}{2}[\omega,\omega]=0$, having as 
basepoint the trivial connection $0$.  If 
$\dim A^1< \infty$, then $\F(A,\g)$ is a Zariski-closed
subset of the affine space $A^1\otimes \g$, which contains 
the closed subvariety $\F^1(A,\g)$  consisting of all tensors of the form 
$\eta \otimes g$ with $d \eta=0$. 

To define the infinitesimal counterpart of characteristic varieties, 
let $\theta\colon \g\to \gl(V)$ be a finite-dimensional representation.  
To each $\omega \in \F(A,\g)$ there is an associated 
covariant derivative, $d_{\omega}\colon A^{\hdot}\otimes V\to 
A^{\hdot+1}\otimes V$, given by $d_{\omega}=d\otimes \id_V +\ad_{\omega}$.  
By flatness, $d_{\omega}^2=0$. 
The {\em resonance varieties}\/ of $A$ with respect to $\theta$ are the sets 
$\RR^i_s(A,\theta)$ consisting of those $\omega \in \F(A,\g)$ 
for which $\dim H^i(A\otimes V, d_{\omega})\ge s$.  
If $A$ is $q$-finite, then these subsets are Zariski-closed in $\F(A,\g)$,
for all $i\le q$. Furthermore, if $H^i(A)\ne 0$, then $\RR^i_1(A,\theta)$ 
contains the closed subvariety $\Pi(A,\theta)$  
consisting of all $\eta \otimes g\in \F^1(A,\g)$ with $\det \theta(g)=0$. 

By work of Arapura \cite{Ar}, the key to understanding the degree-$1$ 
characteristic varieties of rank $1$ for a quasi-projective manifold $M$ 
is the (finite) set $\mathcal{E}(M)$ of regular, surjective maps 
$f\colon M\to S$ for which the generic fiber is connected and the target 
is a smooth curve $S$ with $\chi(S)<0$, up to reparametrization at the 
target.  All such maps extend to regular maps $\bar{f}\colon \overline{M} \to 
\overline{S}$, for some `convenient' compactification   
$\overline{M}=M\cup D$.  Letting $A(M)=A(\overline{M},D)$ 
be the finite $\dga$ model of $M$ constructed by 
Morgan \cite{Mo} and more generally by Dupont \cite{Du},  
and similarly for $A(S)$, we obtain the following structural result 
for the embedded first resonance varieties of a large class of 
quasi-projective surfaces, in the difficult rank $2$ case. 
 
\begin{theorem}
\label{thm:intro7}
Let $M$ be a punctured, quasi-homogeneous surface with 
isolated singularity.  Let $\theta\colon \g\to \gl(V)$ 
be a finite-dimensional representation, where $\g$ is the 
Lie algebra $\sl_2(\C)$ or a Borel subalgebra.
If $b_1(M)>0$, then there is a convenient compactification 
$\overline{M} = M \cup D$  such that 
\begin{align*}
\F(A(M),\g) = \F^1(A(M),\g)\cup 
\bigcup_{f\in \mathcal{E}(M)}  f^{!} (\F(A(S),\g)),
\\
\RR^1_1(A(M),\theta) = \Pi(A(M),\theta)\cup 
\bigcup_{f\in \mathcal{E}(M)}  f^{!} (\F(A(S),\g)).
\end{align*}
\end{theorem}

The restriction on the first Betti number is not essential, as explained 
in Remark \ref{rem:betti0}. To prove Theorem \ref{thm:intro7}, we replace 
$M$ (up to homotopy) by the singularity link, and $A(M)$ by a finite model 
$A$ of this almost free $S^1$-manifold, for which explicit computations are 
done in \S\S\ref{subsec:rk2flat}-\ref{subsec:rk2res}.   We single out in 
Examples \ref{ex:proj} and \ref{ex:w1h10} several more classes 
of quasi-projective manifolds for which the conclusions of the theorem 
hold.  Whether those conclusions hold in complete generality remains 
an open question, which is investigated further in \cite{PS-natjump, PS-noshi}. 

\section{Algebraic models and formality properties}
\label{sect:cdga}

In this section, we collect some relevant facts from 
rational homotopy theory, following \cite{Su77, Mo, FHT, HS, Le}.   
All spaces will be assumed to be path-connected, and the default 
coefficient ring will be a field $\k$ of characteristic $0$. 

\subsection{Differential graded algebras and Hirsch extensions}
\label{subsec:cdga}

The basic algebraic structure we will consider in this work is that 
of a commutative, differential graded algebra (for short, 
a \dga) over the field $\k$.  This is an object of the form 
$A=(A^{\hdot},d)$ where $A$ is a non-negatively graded 
$\k$-vector space, endowed with a multiplication map 
$\cdot\colon A^i \otimes A^j \to A^{i+j}$ satisfying 
$u\cdot v = (-1)^{ij} v \cdot u$,  
and a differential $d\colon A^i\to A^{i+1}$ 
satisfying $d(u\cdot v) =d{u}\cdot v 
+(-1)^{i} u \cdot d{v}$, for all $u\in A^i$ and $v\in A^j$. 

We say that $A$ is {\em connected}\/ if $\dim A^0=1$.  
We also say  that $A$ is {\em $q$-finite}, for some 
$0\le q\le \infty$, if $A$ is connected and 
$\dim \bigoplus_{i\le q}A^{i}<\infty$. When $q=\infty$, 
we will omit it from terminology and notation.  
A morphism $\varphi\colon A\to B$ between two 
$\dga$'s is a {\em quasi-isomorphism}\/ if the 
induced homomorphism in cohomology, $\varphi^*\colon 
H^{\bullet}(A)\to H^{\bullet}(B)$, is an isomorphism.

Denote by $Z^{\hdot} (A)$ the graded vector space of $d$-cocycles of a 
\dga \/ $A$.  Let $P^{\hdot}$ be an oddly-graded, finite-dimensional 
vector space, with homogeneous basis $\{ t_i \in P^{m_i} \}$. Given a 
degree $1$ linear map, $\tau\colon P^{\hdot} \to Z^{\hdot +1} (A)$, 
we define the corresponding {\em Hirsch extension}\/ as the \cdga 
\begin{equation}
\label{eq:hirsch}
\left(A\otimes_{\tau} \bwedge P , d\right), 
\end{equation}
where the differential $d$ extends the 
differential on $A$, while $dt_i=\tau (t_i)$. The following standard 
fact is proved in \cite[Lemmas II.2 and II.3]{Le}, in the case when 
all the degrees $m_i$ are equal.  The same proof gives the result 
in the general case.

\begin{lemma}
\label{lem:hirsch}
The isomorphism type of the \cdga  $(A\otimes_{\tau} \bwedge P, d)$ 
depends only on $A$ and the homomorphism 
induced in cohomology, $[\tau ]\colon P^{\hdot} \to H^{\hdot +1} (A)$. 
Moreover, $[\tau ]$ and $[\tau ]\circ g$ give isomorphic extensions, for any
automorphism $g$ of the graded vector space $P^{\hdot}$.
\end{lemma}

When $\dim P=1$, we have the following natural exact sequence, for all $i$ 
(see e.g. \cite[pp. 178--179]{Hal}):
\begin{equation}
\label{eq=helem}
\xymatrix{
H^{i-m-1}(A) \ar^(.57){[\tau (t)] \cdot }[r] &H^{i}(A)\ar[r] &
H^{i}(A\otimes_{\tau} \bwedge t) \ar[r] & H^{i-m}(A) \ar^{[\tau (t)] \cdot } [r]& H^{i+1}(A) 
}.
\end{equation}

\subsection{Minimal models}
\label{subsec:minmod}

Let $\bwedge (x)$ be the exterior (respectively, polynomial) 
algebra on a single generator $x$ of odd (respectively, even) degree.
A \cdga $\M$ is said to be {\em minimal}\/ 
if the following three conditions are satisfied.

\begin{description}[font=\textit, leftmargin=3.5cm,style=nextline]
\item[Freeness] The underlying commutative graded algebra of $\M$ 
is of the form $\bwedge (\{x_{\alpha}\}) = \bigotimes_{\alpha} \bwedge (x_{\alpha})$, 
with positive-degree variables $\XX=\{x_{\alpha}\}$ indexed by a well-ordered set.
\item[Nilpotence]  For any $\alpha$, we have that 
$d x_{\alpha}\in \bwedge (\{x_{\beta} \mid \beta<\alpha\})$. 
\item[Decomposability]  For any $\alpha$, we have that 
$d x_{\alpha}\in \bwedge^{\! +}(\XX) \cdot \bwedge^{\! +}(\XX)$, where 
$\bwedge^{\! +}(\XX)$ is the ideal generated by $\XX$.
\end{description}

A {\em minimal model}\/ for a \cdga $A$ is a minimal \cdga $\M$ which 
comes equipped with a quasi-isomorphism $\varphi\colon \M\to A$.  
Any \cdga $A$ with $H^0(A)=\k$ has a unique minimal model, 
denoted by $\M(A)$. 
The notion of minimality admits the following refinement.  A minimal 
\cdga $\M$ is said to be {\em $q$-minimal}\/ (for some $q\ge 0$) if 
$\deg(x_{\alpha})\le q$, for all $\alpha$. Such an object is a 
{\em $q$-minimal model}\/ for a \cdga $A$ if there is a 
{\em $q$-equivalence} $\varphi\colon \M\to A$,  i.e., 
the \cdga map $\varphi$ induces a homology isomorphism 
up to degree $q$ and a homology monomorphism in degree $q+1$.  
Again, any \cdga $A$ with connected homology  has a unique $q$-minimal model, 
denoted by $\M_q(A)$. 

In the definition of minimality, let us denote by $V=V^{\hdot}$ the 
graded vector space generated by the set $\XX=\{x_{\alpha}\}$. 
With this notation, $\M(A)=(\bwedge V,d)$ and 
$\M_q(A)=(\bwedge V^{\! \le q},d)$.

\subsection{Formality}
\label{subsec:formality}

Two $\dga$s $A$ and $B$ are said to be {\em weakly equivalent}\/ 
(denoted $A\simeq B$) if there is a zig-zag of quasi-isomorphisms 
connecting $A$ to $B$.  
Likewise, $A$ and $B$ are said to have the same {\em $q$-type}\/ 
(denoted $A\simeq_q B$) if there is a zig-zag of $q$-equivalences 
connecting $A$ to $B$.   If $H^0(A)=\k$, then $A\simeq B$ 
(respectively, $A\simeq_q B$) if and only if $A$ and $B$ share the same 
minimal (respectively, $q$-minimal) model. Clearly, $A\simeq_q B$ implies 
$A\simeq_s B$, for all $s\le q$.

A \cdga $A$ is said to be {\em formal}\/ if it is weakly equivalent 
to the cohomology algebra $H^{\hdot}(A)$, endowed with the 
zero differential. Likewise, $A$ is said to be {\em $q$-formal}\/ 
if $A\simeq_q (H^{\hdot}(A),d=0)$. Plainly, $q$-formality implies 
$s$-formality, for all $s\le q$. 

As shown by Halperin and Stasheff in \cite{HS}, the formality of a \cdga 
with connected, finite-type homology is independent of the 
ground field, with respect to field extensions $\k\subset \K$. 
The same descent property holds for $q$-formality,
for all $q<\infty$, provided $H^{\hdot}(A)$ is $(q+1)$-finite, see \cite{SW}. 
In general, it is easy to check that $A\simeq_q B$
over $\k$ implies  that $A\otimes_\k \K\simeq_q B\otimes_\k \K$, but 
the converse does not necessarily hold.

Now let $X$ be a topological space.  
An important bridge between topology and algebra is provided by the work 
of D.~Sullivan \cite{Su77}, who constructed a \cdga 
$\Omega^{\hdot}(X)= \Omega^{\hdot}(X,\k)$ 
of piecewise polynomial $\k$-forms on $X$, inspired by the de Rham 
algebra of smooth forms on a manifold. This \cdga has the property that 
$H^{\hdot}(\Omega(X,\k))\cong H^{\hdot}(X,\k)$, as graded rings.  

A \cdga $A$ is said to be a {\em model}\/ for $X$ if $A\simeq \Omega(X)$.  
Likewise, $A$ is a {\em $q$-model}\/ for $X$, for some $q\ge 1$, 
if $A\simeq_q \Omega(X)$.  Finally, the space $X$ is said to be formal 
(respectively, $q$-formal) if $\Omega^{\hdot}(X)$ has that property.

\subsection{Formality properties of groups}
\label{subsec:filtform}

Let $X$ be a $1$-finite space, i.e., a 
space homotopy equivalent to a connected CW-complex with only finitely 
many $1$-cells, and let $\pi=\pi_1(X)$.  By considering a 
classifying map $X\to K(\pi,1)$, it is readily seen that the 
$1$-formality of $X$ depends only on $\pi$.  

Let $\m(\pi)$ be the {\em Malcev Lie algebra}\/ of $\pi$ (over $\k$). This is a complete, 
filtered Lie algebra, constructed by Quillen in \cite{Qu}.  As is well-known 
(see e.g.~\cite{PS09, SW}), a finitely-generated group $\pi$ is $1$-formal 
if and only if its Malcev Lie algebra $\m(\pi)$ is isomorphic, as a filtered 
Lie algebra, to the lower central series completion of a quadratic Lie algebra, 
that is, $\m(\pi)\cong \widehat{L}$, where the Lie algebra $L$ is generated 
in degree $1$, and has relations only in degree $2$.

The group $\pi$ is said to be {\em filtered-formal}\/ 
if $\m(\pi)\cong \widehat{L}$, where $L$ is generated 
in degree $1$, and has homogeneous defining relations. 
This property, analyzed in detail in \cite{SW}, is strictly 
weaker than $1$-formality.  For instance, if $\Heis_1$ 
is the $3$-dimensional Heisenberg manifold from Example \ref{ex:heis}, 
then the group $\pi_1(\Heis_1)$ is filtered-formal, yet not $1$-formal. 

\subsection{Hirsch extensions and algebraic $q$-type}
\label{ssqg}

As before, let $(A,d)$ be a \dga, and let $P$ be a graded vector space generated 
by finitely many odd-degree elements $t_i\in P^{m_{i}}$.  Set $m=\max \{ m_i\}$, 
and let $\tau\colon P \to Z^{\le m+1}(A)$ be a degree one homomorphism. 
We then have a Hirsch extension,  $(A\otimes_{\tau} \bwedge P,d)$, 
where the differential $d$ agrees with that of $A$, while taking  
$t_i$ to $\tau(t_i)\in A^{m_{i}+1}$. The next result will be very useful in the sequel.
For $q=\infty$, it follows from more general results on relative Sullivan models, see e.g.
\cite[Lemma 14.2]{FHT}. In our simpler case, we provide an elementary proof, for $q \le \infty$.  

\begin{lemma}
\label{lem=qhirsch}
Suppose $A\simeq_q B$ and $q\ge m+1$.  
There is then a degree one homomorphism $\sigma\colon P\to Z^{\le m+1}(B)$ 
and an identification $H^{\le m+1}(A) \cong H^{\le m+1}(B)$ under 
which $[\sigma]$ corresponds to $[\tau]$, and 
\[
A\otimes_{\tau} \bwedge P \simeq_q B\otimes_{\sigma} \bwedge P.
\]
\end{lemma}

\begin{proof}
Assume first that there is a $q$-equivalence $\varphi \colon A \to B$.   
Since $q\ge m+1$, the induced homomorphism, 
$\varphi^* \colon H^{\le m+1}(A)\to  H^{\le m+1}(B)$,  
is an isomorphism.   Set $\sigma=\varphi\circ \tau$. Then 
$\varphi^*\circ [\tau]=[\sigma]$.  It remains to show that 
$\varphi \otimes \id \colon A\otimes_{\tau} \bwedge P  \to B\otimes_{\sigma} \bwedge P$ 
is a $q$-equivalence.  Since every Hirsch extension may be viewed as a finite 
sequence of elementary extensions with $\dim P=1$, we can argue by induction 
on $\dim P$. The induction step follows from exact sequence  \eqref{eq=helem} 
and the $5$-Lemma.

Now assume that there is a $q$-equivalence $\psi \colon B \to A$.  Arguing 
as above, the map $\psi^* \colon H^{\le m+1}(B)\to  H^{\le m+1}(A)$ 
is an isomorphism.  We may then pick a degree one homomorphism 
$\sigma\colon P\to Z^{\le m+1}(B)$ 
such that $[\sigma]=(\psi^*)^{-1} \circ [\tau]$. In view of Lemma \ref{lem:hirsch}, 
the isomorphism type of the Hirsch extension $B\otimes_{\sigma} \bwedge P$ 
depends only on $B$ and $[\sigma]$, and thus not on the choice of $\sigma$.  
The rest of the argument is exactly as in the first case.

The general case, where $A$ and $B$ are connected by a zig-zag of 
$q$-equivalences, follows from the two particular cases handled above.
\end{proof}

\section{Cohomology jump loci}
\label{sect:cjl}

In this section we review and develop some material 
from \cite{DP-ccm, MPPS, BMPP}.  Unless otherwise 
mentioned, we continue to work over
a fixed field $\k$ of characteristic $0$. 

\subsection{Characteristic varieties}
\label{subsec:cvs}
Let $X$ be a space.  We say that $X$ is {\em $q$-finite}\/ 
if it has the homotopy type of a connected CW-complex with finite $q$-skeleton. 
For $q=\infty$, this means that $X$ is a finite, connected CW-space;  
in this case, we will simply say that $X$ is a finite space.

We will assume henceforth that $X$ is $q$-finite, for some 
$q\ge 1$ (possibly $q=\infty$). In this case, the fundamental 
group $\pi=\pi_1(X)$ is finitely generated. 
Fix a linear algebraic group $G$.  The {\em $G$-representation variety}\/ 
of the discrete group $\pi$ is the set $\Hom_{{\rm gr}} (\pi,G)$, 
viewed in a natural way as an affine algebraic subvariety 
of $G^{\times m}$, where $m$ is the number of generators of $\pi$. 
This variety comes with a distinguished basepoint $1$, the 
trivial representation. 

Now fix also a rational representation 
$\iota \colon G\to \GL(V)$, where $V$ is a finite-dimensional 
$\k$-vector space.  Given a representation $\rho \colon \pi \to G$, 
we may view $V$ as a $\k[\pi]$-module via the action defined 
by $\iota\circ \rho$, and thus consider the cohomology 
groups of $X$ with coefficients in this module. 
The {\em characteristic varieties}\/  of $X$ 
with respect to $\iota$  are the jump loci for these 
twisted cohomology groups.  More precisely, 
in each degree $i\ge 0$ and depth $s\ge 0$, 
define 
\begin{equation}
\label{eq:defv}
\VV^i_s(X,\iota) =\{\rho\in \Hom (\pi,G) \mid \dim_{\k} 
H^i(X , V_{\iota\circ \rho}) \ge s\}.
\end{equation}

These sets are Zariski closed subsets 
of the representation variety, provided $i\le q$. 
Note that $\VV^1_s(\pi,\iota) :=\VV^1_s(X,\iota)$ 
depends only on the fundamental group $\pi$, for all $s$.
In the rank $1$ case, i.e., when $G=\k^*$, $V=\k$, 
and $\iota\colon \k^*\to \GL_1(\k)$ is the standard 
identification, we will simply write these jump loci as 
$\VV^i_s(X)$, and view them as subvarieties of the 
character group $\Hom(\pi_1(X),\k^*)=H^1(X,\k^*)$. 

\subsection{Resonance varieties}
\label{subsec:res}

We now turn to the infinitesimal counterpart of the above 
setup. Let $A=(A^{\hdot}, d)$ be a \dga.  We will assume 
that $A$ is  $q$-finite, for some $q\ge 1$. 
Furthermore, let $\g$ be a finite-dimensional Lie algebra. 
Inside the affine space $A^1\otimes \g$, we shall consider the 
variety $\F(A,\g)$ of {\em $\g$-valued flat connections}\/ on 
$A$, which consists of all solutions of the Maurer--Cartan equation, 
$d\omega +\frac{1}{2}[\omega,\omega]=0$.  This variety is natural in 
both $A$ and $\g$, and has as natural basepoint the trivial flat 
connection, $0$. If $\omega =\sum_i \eta_i \otimes g_i$, with
$\eta_i\in A^1$ and $g_i\in \g$, the flatness condition amounts to
\begin{equation}
\label{eq:flat coords}
\sum_{i} d\eta_i \otimes g_i +
\sum_{i<j} \eta_i \eta_j \otimes [g_i, g_j] =0.
\end{equation}

Now let $\theta\colon \g\to \gl(V)$ be a finite-dimensional representation.  
To each flat connection $\omega \in A^1\otimes \g$ there is an associated 
covariant derivative, $d_{\omega}\colon A^{\hdot}\otimes V\to 
A^{\hdot+1}\otimes V$, given by $d_{\omega}=d\otimes \id_V +\ad_{\omega}$, 
where the adjoint operator $\ad_{\omega}$ also depends on $\theta$.  
Explicitly, if $\omega=\sum_i \eta_i \otimes g_i$, then
\begin{equation}
\label{eq:adv}
d_{\omega}(\alpha\otimes v) = d\alpha \otimes v +
\sum_{i} \eta_i  \alpha \otimes \theta(g_i)(v),
\end{equation}
for all $\alpha\in A^i$ and $v\in V$.
It is readily verified that $d_{\omega}^2=0$, by the flatness condition. 

The {\em resonance varieties}\/ of $A$ with respect to $\theta$  
(in degree $i\ge 0$ and depth $s\ge 0$), are the sets 
\begin{equation}
\label{eq:defr}
\RR^i_s(A,\theta) =\{\omega\in \F (A,\g) \mid \dim_{\k} 
H^i(A\otimes V, d_{\omega}) \ge s\}.
\end{equation}
These sets are Zariski-closed subsets of the variety of flat connections, provided 
that $i\le q$. Let us define the $(q+1)$-truncation of a $\cdga$ $A$ to be the 
quotient $\cdga$ $A^{\le q+1}:= A^{\hdot}/ \bigoplus_{j>q+1}A^j$.
By construction, both $\F (A,\g)$ and $\RR^i_s(A,\theta)$ depend only on 
this truncation, for $i\le q$ and $s\ge 0$.

In the rank $1$ case, i.e., when $\g=\k$  
and $\theta\colon \k\to \gl_1(\k)$ is the standard 
identification, we will omit $\g$ and $\theta$ from the notation. 
In this situation, the variety $\F(A) =\{\omega \in A^1 \mid d\omega =0\}$ 
may be identified with $H^1(A)$, by connectedness of $A$.  
Moreover, the covariant derivative is now given by 
$d_{\omega} \eta =d\eta +\omega \cdot \eta$, for $\eta\in A^{\hdot}$.  
In particular, $\omega$ belongs to $\RR^1_1(A)$ if and only if 
there is an $\eta\in A^1\setminus \k\cdot \omega$ such that 
$d\eta +\omega \cdot \eta=0$.

The rank $1$ resonance varieties obey the following product formula. 

\begin{prop}
\label{prop:prod}
Let $(A,d)$ and $(A',d')$ be two connected $\dga$'s. Then
\[
\RR^q_1(A\otimes A') = 
\bigcup_{i+j=q} \RR^i_1(A) \times \RR^j_1(A').
\]
\end{prop} 

A proof of this statement is given in \cite[Prop.~13.1]{PS-plms} 
under the assumption that both $d$ and $d'$ vanish (see also 
\cite[Prop.~2]{PS-springer}).  The same proof works in this 
wider generality.

\subsection{Germs of complex jump loci}
\label{subsec:germs}
We now take the coefficient field to be $\k=\C$.  
Let $X$ be a $q$-finite space with fundamental group $\pi$, and assume 
$\Omega^{\hdot}(X)$ has the same $q$-type as a $q$-finite \cdga $A$.
Let $\iota\colon G\to \GL(V)$ be a rational representation of
a complex linear algebraic group $G$, and let $\theta=d_1(\iota)\colon \g \to \gl(V)$
be its tangential representation at the identity.  For an affine variety $\XX$ and 
a point $x\in \XX$, we denote by $\XX_{(x)}$ the reduced analytic germ 
of $\XX$ at $x$. 

\begin{theorem}[\cite{DP-ccm}]
\label{thm:germs}
Under the above assumptions, there is an analytic isomorphism 
$\Hom(\pi,G)_{(1)} \isom \F(A,\g)_{(0)}$ which restricts to analytic 
isomorphisms $\VV^i_s(X,\iota)_{(1)} \isom \RR^i_s(A,\theta)_{(0)}$,
for all $i\le q$ and $s\ge 0$.
\end{theorem}

Following \cite{BMSS, Mo}, we say that a \cdga $(A,d)$ over a field 
$\k$ of characteristic zero has {\em positive weights}\/ 
if, for each $i\ge 0$, there is a vector space decomposition, 
$A^i=\bigoplus_{\alpha\in \Z} A^i_\alpha$, such that, if we set $\wt(a)=\alpha$ 
for $a\in A^i_\alpha$, the following conditions hold: 
\begin{enumerate}
\item \label{pw1}
$\wt(a)>0$, for all $a\in A^i_\alpha$ with $i>0$; 
\item  \label{pw2}
$\wt(da)=\wt(a)$, for all $a\in A^i_\alpha$; 
\item  \label{pw3}
$\wt(ab)=\wt(a)+\wt(b)$, for all $a\in A^i_\alpha$ and $b\in A^j_\beta$. 
\end{enumerate}

For instance, if $d=0$, we may set $\wt(a)=i$ for 
all $a\in A^i$.  Moreover, if $(A,d)$ has positive weights and a $d$-cocycle
$e\in A^{m+1}$ is homogeneous with respect to those weights, then the 
Hirsch extension $(A\otimes_e \bwedge (t) , d)$ also has positive 
weights.  Indeed, we simply declare $\wt(t)=\wt(e)$, and extend the weights 
on the Hirsch extension  accordingly. When $d=0$, it follows that any
Hirsch extension $A\otimes_{\tau} \bwedge P$ has positive weights.

Now let $X$ be a space, with Sullivan model $\Omega(X)=\Omega(X,\k)$.  
We say that a $\k$-\cdga $(A,d)$ is a {\em $q$-model with positive 
weights}\/ for $X$ if the following conditions are satisfied:
\begin{enumerate}
\item $A$ is defined over $\Q$;
\item $A$ has positive weights over $\Q$;
\item $A\simeq_q \Omega(X)$, with the isomorphism induced on 
first homology by the zig-zag of $q$-equivalences between 
$A$ and $\Omega(X)$ required to preserve $\Q$-structures.  
\end{enumerate}
As the next result shows, the existence of such a
model for a space $X$ imposes stringent conditions 
on the rank $1$ cohomology jump loci of $X$. 

\begin{theorem}[\cite{DP-ccm}]
\label{thm:mpps-bis}
Let $X$ be a $q$-finite space, and suppose $X$ admits a 
$q$-finite, $q$-model $A$ with positive weights. 
Then all irreducible components 
of $\VV^i_s(X)$ passing through $1$ are algebraic subtori  
of the character group $H^1(X,\C^*)$, for all $i\le q$ 
and $s\ge 0$.
\end{theorem}

\begin{remark}
\label{rem:BW}
In a very recent preprint, Budur and Wang \cite{BW} 
remove the positive weights hypothesis from the above theorem. 
\end{remark}

\begin{remark}
\label{rem:betti0}
When the Betti numbers vanish, the embedded jump loci 
are not interesting, at least from the  
point of view of germs around the origin.  Indeed, 
if $b_1(X)=0$, then $\Hom(\pi_1(X), G)_{(1)}=\{1\}$, by 
\cite[Theorem A]{DP-ccm}. Furthermore, if $b_i(A)=0$ for some 
$i\le q$, then $\RR^i_1(A,\theta)_{(0)}=\emptyset$, by 
\cite[(15)]{MPPS}.
\end{remark}

\subsection{Malcev completions and holonomy Lie algebras}
\label{subsec:malholo}

As before, let $\pi$ be a discrete group, and let $\k$ be a field of 
characteristic $0$. 
The {\em Malcev Lie algebra}\/ of $\pi$ (over $\k$), denoted 
$\m(\pi)$, is a filtered, complete $\k$-Lie algebra whose filtration 
satisfies certain axioms, spelled out by Quillen in \cite[Appendix A]{Qu}. 
In particular, the associated graded Lie algebra of $\m(\pi)$ with 
respect to the aforementioned filtration  is isomorphic 
to the associated graded $\k$-Lie algebra of $\pi$ with respect to the
lower central series (lcs) filtration.  For a detailed treatment of the various 
Lie algebras associated to finitely generated groups we refer to \cite{SW}, 
and the further citations therein. 

Now let $A=(A^{\hdot},d)$ be a $1$-finite \dga \/ over $\k$.   
Following \cite{MPPS}, set $A_i=(A^i)^*$, where $(\cdot)^*$ 
stands for vector space duals, and let $\L^{\hdot}(A_1)$ 
be the free Lie algebra on $A_1$, graded by bracket length.  
Let  $d^*\colon A_2\to A_1=\L^1(A_1)$
and $\cup^*\colon A_2\to A_1\wedge A_1=\L^2(A_1)$ 
be the dual maps to the differential and the product map, 
respectively.  
The  {\em holonomy Lie algebra}\/ of $A$
is the finitely presented Lie algebra
\begin{equation}
\label{eq:hololie}
\h(A) = \L(A_1) / \ideal(\im (d^*+\cup^*)).  
\end{equation} 
Clearly, this construction is functorial, and depends only 
on the $2$-truncation $A^{\le 2}$.

The following theorem strengthens a result proved by 
Bezrukavnikov \cite{Bez} in the case when $A$ is a
quadratic algebra. 

\begin{theorem}[\cite{BMPP}, Theorem 3.1]
\label{thm:malholo}
Let $X$ be a $1$-finite space, and 
suppose $A$ is a $1$-finite $1$-model for $X$, defined over $\k$. 
Then $\m(\pi_1(X))$ is isomorphic (as a filtered Lie algebra) 
to the lcs completion of $\h(A)$, over $\k$.
\end{theorem}

The next result relates the holonomy Lie algebra to the germs 
of the representation varieties considered in \S\ref{subsec:germs}. 

\begin{corollary}
\label{cor:sophus}
Let $\pi=\pi_1(X)$ be the fundamental group of a $1$-finite 
space $X$, and let $G$ be a complex linear algebraic group, with 
Lie algebra $\g$. 
If $A$ is a $1$-finite $1$-model for $X$ over $\C$, then the 
analytic germs $\Hom_{\rm gr}(\pi,G)_{(1)}$ and 
$\Hom_{\rm Lie}(\h(A),\g)_{(0)}$ are isomorphic. 
\end{corollary}

\begin{proof}
Proposition 4.5 from \cite{MPPS} provides a natural isomorphism 
of affine varieties between $\F(A,\g)$ and  $\Hom_{\rm Lie}(\h(A), \g)$. 
The claim then follows from Theorem \ref{thm:germs}.
\end{proof}

Corollary \ref{cor:sophus} extends Theorem 17.1 from \cite{KM}, 
from the $1$-formal case to the much broader class of
groups having a $1$-finite $1$-model.

\section{Models for compact Lie group actions: Part I}
\label{sect:Kmod1}

In this section, we first construct from manifolds admitting finite models new examples 
of manifolds supporting an almost free action by a compact Lie group, which also have
finite models. Secondly, we begin an algebraic study of this construction, related to 
the behavior of additional properties (existence of positive weights and formality).
This leads to new results on the partial formality properties of such manifolds, 
and related spaces. 

\subsection{Almost free actions and Hirsch extensions}
\label{subsec:s1act}

From now on, $K$ will denote a compact, connected, real Lie group. 
Consider the universal principal $K$-bundle, 
\begin{equation}
\label{eq:bundle}
\xymatrix{K\ar[r] & EK \ar[r] &BK},
\end{equation}
with contractible total space $EK$ and with base 
space the classifying space for $K$-bundles, $BK=EK/K$.  
By a classical result of Hopf, the cohomology ring of $K$ (with 
coefficients in a characteristic zero field $\k$) is isomorphic to the 
cohomology ring of a finite product of odd-dimensional spheres.  That is,
\begin{equation}
\label{eq:hopf}
H^{\hdot}(K) \cong \bwedge P^{\hdot}
\end{equation}
where $P^{\hdot}$ is an oddly-graded, finite-dimensional vector space, 
with homogeneous basis $\{ t_{\alpha} \in P^{m_{\alpha}} \}$, 
for some odd integers $m_1,\dots ,m_r$, where $r=\rank(K)$.  

Now let $M$ be a compact, connected, differentiable manifold 
on which the compact, connected Lie group $K$ acts smoothly. 
Both $M$ and the orbit space $N=M/K$ are finite spaces. For $N$,
this follows from triangulability of stratified spaces \cite{Gor, V}.
We consider the diagonal action 
of $K$ on the product $EK\times M$, and form the Borel construction, 
\begin{equation}
\label{eq:borel}
M_K=(EK\times M)/K.
\end{equation} 
Let $\proj \colon M_K \to N$ be the map induced 
by the projection $\proj_2 \colon EK \times M \to M$. 

The $K$-action on $M$ is said to be {\em almost free}\/ if all its 
isotropy groups are finite. When this assumption is met, work of 
Allday and Halperin \cite{AH} provides a very useful 
Hirsch extension model for the manifold $M$.  

\begin{theorem}[\cite{AH}]
\label{thm:rs}
Suppose $M$ admits an almost free $K$-action, with orbit space $N=M/K$.  
There is then a map $\sigma\colon P^{\hdot} \to Z^{\hdot +1}(\Omega (N) )$ 
such that $\proj^*\circ [\sigma]$ is the transgression in the principal bundle 
$K\to EK\times M \to M_K$, and 
\[
\Omega (M) \simeq \Omega (N) \otimes_{\sigma} \bwedge P\, .
\]
\end{theorem}

This theorem may be applied for instance to the total space $M$ 
of a principal $K$-bundle over a compact manifold $N=M/K$.

The next result adds a new interesting class of finite spaces 
having finite $\cdga$ models to the known examples from 
\cite{DP-ccm} and \cite{K-jdg}. 

\begin{lemma}
\label{lem:newfinite}
Let $M$ be an almost free $K$-manifold. Denote by $m$ the maximum 
degree of $P^{\hdot}$, where $\bwedge P^{\hdot} = H^{\hdot}(K)$.
\begin{enumerate}
\item \label{nf1}
Suppose $B$ is a $q$-finite $q$-model of the orbit space $N=M/K$, 
with $q\ge m+1$.  Then a suitable Hirsch extension
$A=B\otimes_{\tau} \bwedge P$ is  a $q$-finite $q$-model for $M$. 
\item \label{nf2}
Suppose $N=M/K$ is $q$-formal.  Then we may take 
$B^{\hdot}=(H^{\hdot}(N), d=0)$, and $A=B\otimes_{\tau} \bwedge P$  
is  a $q$-finite $q$-model of $M$ with positive weights.
\item \label{nf3}
Under the same formality assumption, all irreducible components of 
$\VV^i_s(M)$ containing the origin are algebraic subtori of 
the character group $H^1(M, \C^*)$, for all $i\le q$ and $s\ge 0$. 
\end{enumerate}
\end{lemma}

\begin{proof}
Claim \eqref{nf1} and the first part of claim \eqref{nf2} follow 
from Theorem \ref{thm:rs} and Lemma \ref{lem=qhirsch}. 
The other claims follow from Theorem \ref{thm:mpps-bis} 
and the discussion preceding it.
\end{proof}

In the case of principal $K$-bundles, we can say more. 
Write $H^{\hdot}(K, \Q) \cong \bwedge P_K^{\hdot}= \bwedge (t_1,\dots,t_r)$, 
as before.

\begin{theorem}
\label{thm:krealiz}
Let $N$ be a finite space and $K$ be a compact connected real Lie group.
If $N$ has a finite model $B$ over $\Q$, then any Hirsch extension 
$A=B\otimes_{\tau} \bwedge P_K$ can be realized as a finite model 
over $\Q$ of some principal $K$-bundle $M$ over $N$. When $B$ 
has positive weights and the image of $[\tau]$ is generated by 
weighted-homogeneous elements, $A$ also has positive weights. 
\end{theorem}

\begin{proof}
Set $e_{\alpha}=[\tau (t_{\alpha})] \in H^{\hdot}(B) \equiv H^{\hdot}(N, \Q)$. 
By \cite[Corollary 3.4]{Pa}, there is a principal $K$-bundle $M$ over $N$ 
having characteristic classes $\lambda_{\alpha}\cdot e_{\alpha}$, $\alpha =1,\dots,r$,
for some $\lambda_{\alpha} \in \Q^{\times}$. The fact that $A$ is a finite model of 
$M$ over $\Q$ follows from the Hirsch Lemma.  The discussion on positive 
weights from \S \ref{subsec:germs} completes the proof.
\end{proof}

This result provides a systematic way of constructing new examples of finite 
spaces with finite models, with interesting non-nilpotent topology. It also gives 
a simple criterion which ensures that the existence of positive weights on old models
passes to the new models. It is known that the stronger formality property may 
not be inherited by Hirsch extensions of a formal $\dga$.  In the next subsection, 
we provide a criterion which solves this issue, for partial formality.

\subsection{Graded regularity and partial formality}
\label{subsec:partform-circle}

Fix an integer $q\ge 0$.  Our next theorem uses graded versions 
of two classical notions from commutative algebra.    Although the 
analogous result for $q=\infty$ is well-known (cf.~\cite{H87}), we shall 
need the precise form stated here, as a crucial ingredient in the proofs 
of Theorems \ref{thm:highf} and \ref{thm:highf-bis}  below.

Let $H^{\hdot}$ be a connected commutative graded algebra over a characteristic 
zero field $\k$. We say that a homogeneous element $e\in H^k$ is a non-zero 
divisor up to degree $q$ if the multiplication map $e \cdot \colon H^i \to H^{i+k}$ 
is injective, for all $i\le q$.  (For $q=0$, this simply means that $e\ne 0$.) 

Likewise, we say that a sequence $e_1,\dots ,e_r$
of  homogeneous elements in $H^+$ is {\em $q$-regular}\/ 
if the class of each $e_{\alpha}$ is a non-zero divisor up to 
degree $q-\deg(e_{\alpha})+2$ in the quotient ring $H/\sum_{\beta <\alpha} e_{\beta} H$.
(This implies in particular that the elements  $e_1,\dots ,e_r$ are 
linearly independent over $\k$, when  $q\ge \deg(e_{\alpha})-2$ for all $\alpha$.)

\begin{theorem}
\label{thm:lefred}
Suppose $e_1,\dots ,e_r$ is an even-degree, $q$-regular sequence in $H^{\hdot}$.  
Then the Hirsch extension $A=H \otimes_{\tau} \bwedge (t_1,\dots ,t_r)$ with 
$d=0$ on $H$  and $dt_{\alpha}=\tau (t_{\alpha})=e_{\alpha}$ has 
the same $q$-type as $(H/\sum_{\alpha} e_{\alpha} H, d=0)$.
In particular, $A$ is $q$-formal.
\end{theorem}

\begin{proof} 
The canonical projection $H\surj H/\sum_{\alpha} e_{\alpha} H$ 
extends to a morphism of graded algebras, 
$\psi \colon H \otimes \bwedge (t_1,\dots ,t_r) \to H/\sum_{\alpha} e_{\alpha} H$, 
by setting $\psi (t_{\alpha})=0$. In turn, this morphism defines a $\dga$ map, 
$\psi \colon H \otimes_{\tau} \bwedge (t_1,\dots ,t_r) \to 
(H/\sum_{\alpha} e_{\alpha} H, d=0)$.
We will show that $\psi$ is a $q$-equivalence. 

Let $\varphi \colon (H, d=0) \inj A$ be the canonical $\dga$ inclusion,
inducing $\varphi^*\colon  H^{\hdot}/\sum_{\alpha} e_{\alpha} H 
\to H^{\hdot} (A)$. Clearly, $\psi^*$ is surjective and $\psi^* \circ \varphi^*=\id$. 
It follows that $\psi$ is a $q$-equivalence if and only if the map 
$\varphi^*\colon H^{i} \to H^{i} (A)$ is surjective for all $i\le q+1$.

We prove this by induction on the length of our $q$-regular sequence. 
Denote by $m_{\alpha}$ the degree of $t_{\alpha}$. 
Set $A_{\alpha}=H \otimes_{\tau} \bwedge 
(t_1,\dots ,t_{\alpha})$ and note that $A_{\alpha +1}=A_{\alpha} 
\otimes_{e_{\alpha +1}} \bwedge (t_{\alpha +1})$. Suppose that
the map $\varphi_{\alpha}^*$ is surjective up to degree $q+1$. 
To obtain the same property for $\varphi_{\alpha +1}$
at the induction step, it suffices to show that the 
map $H^{i} (A_{\alpha}) \to H^{i} (A_{\alpha +1})$ 
is surjective, for all $i\le q+1$. In view of \eqref{eq=helem}, 
this is equivalent to showing that the map 
$[e_{\alpha +1}] \cdot \colon H^{i-m_{\alpha +1}} (A_{\alpha}) \to H^{i+1} (A_{\alpha})$ 
is injective, for all $i\le q+1$.

Consider the following commuting diagram 
\begin{equation}
\label{eq:cd}
\begin{gathered}
\xymatrix{
H^{i-m_{\alpha +1}}/ \sum_{\beta \le \alpha} e_{\beta}H 
\ar^(.59){\varphi_{\alpha}^*}[r] \ar@{^{(}->}^{e_{\alpha +1}\cdot}[d]
&H^{i-m_{\alpha +1}} (A_{\alpha}) \ar^{[e_{\alpha +1]}\cdot}[d]\\
H^{i+1}/ \sum_{\beta \le \alpha} e_{\beta}H & H^{i+1} (A_{\alpha}) \, .
\ar_(.4){\psi_\alpha^*}[l]
}
\end{gathered}
\end{equation}

Since we are assuming that $i\le q+1$, the $q$-regularity of our sequence 
implies that the map $e_{\alpha +1}$ on the left side of \eqref{eq:cd} is injective. 
By our induction hypothesis, the map $\varphi_{\alpha}^*$ is
an isomorphism. By commutativity of \eqref{eq:cd}, the map 
$[e_{\alpha +1}] \cdot$ must be injective, 
and this completes the proof.
\end{proof}

Classical theorems of Borel \cite{Borel} and Chevalley \cite{Chev} 
provide a machine for constructing graded algebras which satisfy the hypothesis 
of Theorem \ref{thm:lefred}, in the case when $q=\infty$.

\begin{construction}
\label{con:regseq}
Let $H^{\hdot}(BK)$ be the cohomology algebra of the classifying 
space of a compact, connected Lie group $K$.  Let $T$ be a maximal 
torus in $K$, and let $W=NT/T$ be the Weyl group. The classifying 
space $BT$ is the product of $r$ copies of $\CP^{\infty}$, where $r$ 
is the rank of $K$. Its cohomology algebra is $H^{\hdot}(BT)=\k [x_1,\dots, x_r]$,
with degree $2$ free algebra generators, on which $W$ acts by graded 
algebra automorphisms. 

The natural map $\kappa \colon BT \to BK$ identifies the cohomology 
algebra $H^{\hdot}(BK)$ with the invariant subalgebra
of the $W$-action. More precisely, $H^{\hdot}(BK)$ is isomorphic 
to a polynomial ring of the form $\k [f_1,\dots, f_r]$, 
where each $f_{\alpha}$ is a $W$-invariant polynomial of 
even degree $m_{\alpha}+1$, with $m_{\alpha}$ 
as in \S\ref{subsec:s1act}. Moreover, $f_1,\dots , f_r$  
forms a regular sequence in $\k [x_1,\dots, x_r]$. 
\end{construction}

Let $U\subseteq K$ be a closed, connected subgroup of a compact, 
connected Lie group. As shown in \cite{Th}, the Sullivan minimal model 
of the homogeneous space $K/U$ is a Hirsch extension of the form 
$A=H \otimes_{\tau} \bwedge (t_1,\dots ,t_s)$, where $H^{\hdot}$ 
is a free graded algebra on finitely many even-degree generators, 
with zero differential, as in Theorem \ref{thm:lefred}.  
As is well-known, not all homogeneous spaces $K/U$ are formal. 
Nevertheless, the criterion from the aforementioned theorem 
may be used to obtain valuable information on their partial 
formality properties.

\begin{example}
\label{ex=homnonf}
For the homogeneous space $\Sp(5)/\SU(5)$, 
the aforementioned algebra $H^{\hdot}$ has two free generators,
$x_6$ and $x_{10}$, where subscripts denote degrees, and the sequence 
from Theorem \ref{thm:lefred} is $\{x^2_6, x^2_{10}, x_6 x_{10}\}$, 
see \cite[Exercise 3.5, p.~\!143]{FOT}. It is straightforward to deduce 
from our theorem that $\Sp(5)/\SU(5)$ is $19$-formal. 
On the other hand, an easy computation with Massey triple 
products shows that this estimate is sharp, that is, 
$\Sp(5)/\SU(5)$ is not $20$-formal. 
\end{example}

\subsection{Partial formality of $K$-manifolds}
\label{subsec:pfk}

Now let $M$ be an almost free $K$-manifold. Using  the setup from 
\S\ref{subsec:s1act}, write $H^{\hdot}(K)= \bwedge (t_1,\dots ,t_{r})$, 
and denote the transgression of $t_{\alpha}$ by 
$e_{\alpha} \in H^{m_{\alpha} +1} (M/K)$.  As before, set 
$m= \max \{ m_{\alpha} \}$.

\begin{theorem}
\label{thm:nformal}
Suppose the $K$-action on $M$ is almost free, the orbit space $N=M/K$ is $k$-formal, 
for some $k\ge m+1$, and $e_1,\dots ,e_r$ is a $q$-regular sequence in $H^{\hdot}(N)$,
for some $q\le k$.  
Then
\[
\Big(H^{\hdot}(N) \big/\sum_{\alpha=1}^r e_{\alpha} H^{\hdot}(N)\big., d=0\Big)
\] 
is a finite $q$-model for $M$.  In particular, $M$ is $q$-formal.
\end{theorem}

\begin{proof}
Since the action of $K$ on $M$ is almost free, 
Theorem \ref{thm:rs} insures that the Hirsch extension  
$\Omega (N) \otimes_{\tau} \bwedge P$ is a model for $M$. 
The $k$-formality assumption on $N$ means that 
there is a zig-zag of $k$-equivalences between $\Omega (N)$ and $H(N)$.
It follows from \cite[Proposition 3.1]{Ma} that this zig-zag
may be chosen to induce the identity in homology, in degrees up to $k$.   

Now, since  $k\ge m+1$ and $k\ge q$,  Lemma \ref{lem=qhirsch} 
insures that $\Omega (N) \otimes_{\tau} \bwedge P$ has the same $q$-type as 
$H (N)\otimes_{\sigma} \bwedge (t_1,\dots ,t_r)$, 
where $[\sigma]$ corresponds to $[\tau]$ under the 
induced isomorphism $H^{\le m+1}(N) \cong H^{\le m+1}(\Omega(N))$.
By the above, this isomorphism is the identity, and thus 
$\sigma(t_{\alpha})= e_{\alpha}$, for all $\alpha$.  

Using now the $q$-regularity assumption on the sequence $\{e_\alpha\}$, 
Theorem \ref{thm:lefred} applies, and we obtain the desired conclusion.
\end{proof}

As illustrated in the next two examples, the $q$-regularity 
assumption from Theorem \ref{thm:nformal} is
optimal with respect to the $q$-formality conclusion for 
the manifold $M$, when $K=S^1$ or $S^3$. 

\begin{example}
\label{ex:heis-bis}
Let $M=\Heis_1$ be the $3$-dimensional Heisenberg nilmanifold 
from Example \ref{ex:heis}.  This manifold is the total space of the 
principal $S^1$-bundle over the formal manifold $N=S^1 \times S^1$, 
with Euler class $e\in H^2 (N)$ equal to the orientation class. 
In this case, the sequence $\{ e \}$ is $0$-regular, but not $1$-regular
in $H^{\hdot} (N)$. In fact, as mentioned previously, $M$ is not $1$-formal. 
As explained in Example \ref{ex:heis}, this is the first manifold in a series, 
$\Heis_n$, where $(n-1)$-regularity implies $(n-1)$-formality in an optimal way. 
\end{example}

\begin{example}
\label{ex:hopf-nonformal}
Let  $M$ to be the total 
space of the principal $S^3$-bundle over $N=S^2 \times S^2$ 
obtained by pulling back the Hopf bundle $S^7\to S^4$ along 
a degree-one map $N\to S^4$. As above, $N$ is 
formal, and the Euler class $e\in H^4 (N)$ 
is the orientation class.  In this case, $\{ e \}$ is 
$3$-regular, but not $4$-regular 
in $H^{\hdot} (N)$, and Theorem \ref{thm:nformal} 
says that $M$ is $3$-formal. Direct computation with the 
minimal model of $M$ shows that, in fact, $M$ is 
not $4$-formal.
\end{example}

In general, though, our lower bound for the $q$-formality of manifolds 
with almost free $K$-actions is not sharp. We use Construction \ref{con:regseq} 
to illustrate this point.  

\begin{example}
\label{ex=notsharp}
Let $K$ be a compact, connected Lie group of rank $r$, and identify 
$H^{\hdot}(BK)$ with $\k [f_1,\dots, f_r]$. For any $n\ge 1$, let $M_n$ be the 
principal $K$-bundle over the formal manifold $N=(\CP^n)^{\times r}$
classified by the restriction of the natural map $\kappa \colon BT \to BK$ 
to $N$. It is immediate to check that the inclusion $j\colon N \inj BT$
induces an isomorphism in cohomology, up to degree $2n+1$. 
Furthermore, it follows from Construction \ref{con:regseq} that  
the sequence $e_{\alpha}=j^*(f_{\alpha})$, $\alpha=1,\dots, r$  
is $(2n-1)$-regular in $H^{\hdot}(N)$. Therefore, 
by Theorem \ref{thm:nformal}, the manifold $M_n$ is $(2n-1)$-formal. 

On the other hand, the Hirsch extension $H^{\hdot}(BT) \otimes_{f} \bwedge P_K$
is fully formal, by the classical case $q=\infty$ of Theorem \ref{thm:lefred}. 
The argument from Lemma \ref{lem=qhirsch} shows that the $\dga$ map 
$j^*\otimes \id \colon H^{\hdot}(BT) \otimes_{f} \bwedge P_K 
\to H^{\hdot}(N) \otimes_{e} \bwedge P_K$ is a $2n$-equivalence. 
Hence, $M_n$ is actually $2n$-formal.
\end{example}

\section{Models for compact Lie group actions: Part II}
\label{sect:Kmod2}

We continue with the setup from the previous section, and investigate 
several algebraic and geometric objects associated to a closed 
manifold $M$ endowed with an almost free action by a compact 
Lie group $K$, to wit, the Malcev Lie algebra and the representation 
varieties of $\pi_1(M)$, as well as the rank $1$ cohomology 
jump loci of $M$.

\subsection{Malcev completion and representation varieties}
\label{subsec:formalbase}

Let $H$ be a $2$-finite \cdga with zero differential, and 
let $A=H\otimes_{\tau} \bwedge P$ be a Hirsch extension, 
where $P$ is an oddly-graded, finite-dimensional vector space.

\begin{theorem}
\label{thm:filtf}
The holonomy Lie algebra $\h(A)$ admits a finite presentation
with generators in degree $1$ and relations in degrees $2$ and $3$. 
\end{theorem}

\begin{proof}
Clearly,  the $\dga$s $H\otimes_{\tau} \bwedge P$ and $H\otimes_{\tau} \bwedge P^1$
 have the same $2$-truncation; hence, we may assume that $P=P^1$. Pick bases  
$\{ t_{\alpha} \} \cup \{ s_{\beta} \}$ for $P^1$ and $\{ e_{\beta} \} \cup \{ f_{\gamma}\}$ 
for $H^2$ such that $d_A( t_{\alpha})=0$ and $d_A (s_{\beta})= e_{\beta}$. Let $\{ h_i \}$ 
be a basis for $H^1$, and denote by $\{ (\cdot)^* \}$ the dual bases. Plainly, 
$A^1= H^1 \oplus P$ and $A^2=H^2 \oplus \bwedge^2 P \oplus (H^1 \otimes P)$. 

By construction, the map $d_A^*$ is zero on $A_2$, except for 
$d_A^* (e_{\beta}^*)= s_{\beta}^*$.  Moreover, we have the decomposition 
$\bwedge^2 A^1 = \bwedge^2 H^1 \oplus \bwedge^2 P \oplus (H^1 \otimes P)$. 
Again by construction, $\cup_A=\cup_H$ on the first summand, and $\cup_A=\id$ 
for the others. Set $u_{\beta}= \cup_H^*  e_{\beta}^* \in \L^2 (h_i^*)$ and 
$v_{\gamma}= \cup_H^*  f_{\gamma}^* \in \L^2 (h_i^*)$. 

By \eqref{eq:hololie}, the Lie algebra $\h(A)$ is generated 
by $\{ t_{\alpha}^* \} \cup \{ s_{\beta}^* \} \cup \{ h_i^* \}$,
with the following defining relations: $s_{\beta}^* + u_{\beta}= 0\,  (\forall \beta)$; 
$v_{\gamma}= 0\, (\forall \gamma)$;
$[t_{\alpha}^*, t_{\alpha'}^*]=0\, (\forall \alpha, \alpha')$; 
$[s_{\beta}^*, s_{\beta'}^*]=0\, (\forall \beta, \beta')$; 
$[t_{\alpha}^*, s_{\beta}^*]=0\, (\forall \alpha, \beta)$; 
$[h_i^*, t_{\alpha}^*]=0\, (\forall i, \alpha)$;  and 
$[h_i^*, s_{\beta}^*]=0\, (\forall i, \beta)$. 

Using the first batch of relations to eliminate the generators $\{ s_{\beta}^* \}$, we find 
that $\h(A)$ is generated in degree $1$ by $\{ t_{\alpha}^* \}  \cup \{ h_i^* \}$, and has 
the following relators:
\addtocounter{equation}{-11}
\renewcommand{\theequation}{\Roman{equation}}
\begin{align}
\label{eq=hol21}
&v_{\gamma} && (\forall \gamma),
\\
\label{eq=hol22}
&[t_{\alpha}^*, t_{\alpha'}^*] && (\forall \alpha, \alpha'),
\\
\label{eq=hol23}
&[h_i^*, t_{\alpha}^*] && (\forall i, \alpha),
\\
\label{eq=hol32}
&[h_i^*, u_{\beta}] &&(\forall i, \beta),
\\
\label{eq=hol31}
&[t_{\alpha}^*, u_{\beta}] &&(\forall \alpha, \beta),
\\
\label{eq=hol4}
&[u_{\beta}, u_{\beta'}] && (\forall \beta, \beta').
\end{align}
\addtocounter{equation}{5}

We claim that relations \eqref{eq=hol4} and \eqref{eq=hol31} follow from the others.
Indeed, $u_{\beta'} \in \L^2 (H_1)$  and, for any $i,j$, the relation 
$[[h_i^*, h_j^*], u_{\beta}]= [h_i^*, [h_j^*, u_{\beta}]]- [h_j^*, [h_i^*, u_{\beta}]]$
is a consequence of \eqref{eq=hol32}. Similarly, relations  \eqref{eq=hol31} 
may be eliminated using \eqref{eq=hol23}.

Of the remaining relations, \eqref{eq=hol21}--\eqref{eq=hol23} 
are quadratic, while \eqref{eq=hol32} are cubic. 
This completes the proof.
\end{proof}

\begin{corollary}
\label{cor:ffcircle}
Suppose $M$ supports an almost free $K$-action with $2$-formal orbit space. 
Then:
\begin{enumerate}
\item \label{ff1}
The group $\pi=\pi_1(M)$ is filtered-formal. 
More precisely, the Malcev Lie algebra $\m(\pi)$ is isomorphic 
to the lcs completion of the quotient $\L(H_1(\pi, \k))/\mathfrak{r}$, where 
$\mathfrak{r}$ is a homogeneous ideal generated in degrees $2$ and $3$. 

\item \label{ff2}
For every complex linear algebraic group $G$,  
the germ at the origin of the representation variety 
$\Hom_{\gr}(\pi,G)$ is defined by quadrics and cubics only.
\end{enumerate}
\end{corollary}

\begin{proof}
By Theorem \ref{thm:rs}, the Hirsch extension 
$\Omega (N) \otimes_{\tau} \bwedge P^1$ is a $1$-model for $M$.
In view of Lemma \ref{lem=qhirsch}, the Hirsch extension  
$A^{\hdot}=H^{\hdot} (N) \otimes_{\tau} \bwedge P^1$ is 
a $1$-finite $1$-model of $M$. Claim \eqref{ff1} now follows 
from Theorems \ref{thm:malholo} and
\ref{thm:filtf}.  Finally, claim \eqref{ff2} is 
a consequence of Corollary \ref{cor:sophus} and Theorem \ref{thm:filtf}.
\end{proof}

The second statement in the above corollary is analogous to the 
quadraticity obstruction for fundamental groups of compact K\"ahler 
manifolds obtained by Goldman--Millson in \cite[Theorem 1]{GM}. Clearly, 
the corollary applies to principal $K$-bundles over formal manifolds.

\subsection{Flat connections and representation varieties}
\label{subsec:g-flat}

We recall from \cite{MPPS} 
that, for any $1$-finite \cdga $A$ and any finite-dimensional Lie algebra 
$\g$, the set 
\begin{equation}
\label{eq:f1}
\F^1(A,\g)=\{ \omega = \eta\otimes g \in A^1\otimes \g \mid d \eta=0\}
\end{equation}
is a Zariski closed, homogeneous subvariety of the 
variety of $\g$-valued flat connections, 
$\F(A,\g)$. The subvariety $\F^1(A,\g)$ may be called 
the `trivial part' of $\F(A,\g)$, since it is
completely determined by the vector spaces $H^1(A)$ and $\g$. 

Now suppose $A=B\otimes_{\tau} \bwedge P$ is a Hirsch 
extension of a $1$-finite \cdga  $(B, d)$, and 
let $\varphi\colon B\inj A$ be 
the canonical $\cdga$ inclusion. 
By naturality, we have an inclusion, 
\begin{equation}
\label{eq:flataincl}
\F(A,\g)\supseteq \F^1(A,\g) \cup \varphi^{!} (\F(B,\g)).
\end{equation}

\begin{prop}
\label{prop:flataeq}
If $\g$ is a Lie subalgebra of $\sl_2$, then \eqref{eq:flataincl} 
becomes an equality.
\end{prop}

\begin{proof}
Key to our proof is the easily checked fact that two matrices 
$g,g' \in \sl_2$ commute if and only if $\rank \{ g,g' \} \le 1$. 
Pick any $\omega \in \F(A,\g)\setminus \F^1(A,\g)$. We have to show that
$\omega \in \varphi^{!} (\F(B,\g))$. Write $\omega =\eta +\eta'$, with
$\eta \in B^1\otimes \g$ and $\eta' \in P^1\otimes \g$. Note that 
\[
A^2= B^2 \oplus \bwedge^2 P^1 \oplus (B^1 \otimes P^1).
\]
Under this decomposition, the three
components of the Maurer--Cartan equation \eqref{eq:flat coords} are:
\[
d \eta'+ d \eta + \tfrac{1}{2} [\eta, \eta]=0; \quad [\eta', \eta']=0; \quad [\eta, \eta']=0.
\]

Write $\eta =\sum_i h_i \otimes g_i$ and 
$\eta' =\sum_{\alpha} t_{\alpha} \otimes g'_{\alpha}$,
with respect to some fixed bases for $B^1$ and $P^1$. 
The last two components of the flatness 
condition for $\omega$ then become
\[
[g'_{\alpha}, g'_{\beta}]=0, \, \forall \, \alpha, \beta; \quad
[g'_{\alpha}, g_{i}]=0, \, \forall \, \alpha, i.
\] 
If $\eta' \ne 0$, these conditions force $\omega \in \F^1(A,\g)$,
a contradiction. Hence, $\eta' = 0$. The first component of the flatness 
condition for $\omega$ then becomes $d \eta + \tfrac{1}{2}[\eta, \eta]=0$. Therefore, 
$\omega = \varphi^{!} (\eta)$, with $\eta \in \F(B,\g)$, and this 
completes the proof.
\end{proof}

As we shall see later on, the property from the above proposition has the 
same flavor as similar results which hold in the context of smooth complex 
algebraic varieties. Proposition \ref{prop:flataeq} also has the following 
topological interpretation. 

\begin{remark}
\label{rem=kpencil}
Assume that $K$ acts almost freely on $M$, 
with formal orbit space $N$.  Working over $\C$, we have that $(H^{\hdot} (N), d=0)$
 is a finite model for $N$, while $H^{\hdot} (N)\otimes_{\tau} \bwedge P$
is a finite model for $M$, by Lemma \ref{lem:newfinite}. Moreover, the 
morphism $\varphi\colon H\inj A$ models the canonical projection,
$p\colon M\to N$. Let $G$ be either $\SL_2 (\C)$ or a Borel subgroup, 
with Lie algebra $\g$.   
Theorem \ref{thm:germs} and Proposition \ref{prop:flataeq} then say that  
`the non-trivial part of $\Hom_{\gr}(\pi_1 (M),G)_{(1)}$ pulls back from $N$, 
via the map $p$.'  We refer to \cite{PS-natjump} for the precise description of 
this representation variety near $1$, in the case when $M$ is a principal 
$K$-bundle over $N$.
\end{remark}

\subsection{Rank $1$ cohomology jump loci}
\label{sect:res1sas}

We now turn our attention to the jump loci associated to algebraic models.  
We start with rank $1$ resonance. 

\begin{prop}
\label{prop:circleres}
Let $B$ be a connected \dga. Fix an element $e\in B^2$ with $de=0$, and 
let $A=(B\otimes_{e} \bwedge (t), d)$ be the corresponding Hirsch extension. 
\begin{enumerate}
\item \label{e0}
If $[e]=0$, then $\RR^i_1(A) =  \RR^{i-1}_1(B) \cup \RR^i_1(B)$, for all $i$. 
\item  \label{e1}
If $[e]\ne 0$, then
\begin{enumerate}
\item \label{r1}
$\RR^i_s (A)\subseteq \RR^{i-1}_1 (B) \cup \RR^i_s (B)$, for all $i$ and $s$.
\item \label{r2}
$\RR^1_s (A)=\RR^1_s (B)$, for all $s$.
\end{enumerate}
\end{enumerate}
\end{prop} 

\begin{proof}
First suppose $[e]=0$.  Then $A$ is isomorphic to $B\otimes (\bwedge (t), d=0)$.  
Applying the product formula for resonance varieties from 
Proposition~\ref{prop:prod}, claim \eqref{e0} follows. 

Now suppose $[e]\ne 0$. 
Denoting by $\varphi \colon B\to A$ the canonical \cdga inclusion,  
an easy computation shows that $H^1(\varphi) \colon H^1(B) \to  H^1(A)$ 
is an isomorphism, since both $A$ and $B$ are connected. Thus, the 
varieties of rank $1$ flat connections on $A$ and $B$ may be identified. 

For $\omega \in \F (A)\equiv \F (B)$, the cochain complex $(B^{\hdot}, 
d_{\omega})$ is clearly a subcomplex of $(A^{\hdot}, d_{\omega})$. 
Using the description of $d_{\omega}$ from \S \ref{subsec:res}, 
it is straightforward to check that the quotient complex can be 
identified with $(B^{\hdot -1}, -d_{\omega})$. Claim \eqref{r1} 
then follows by examining the associated long exact
sequence in homology,
\begin{equation}
\label{eq:les1}
\xymatrixcolsep{20pt}
\xymatrix{
\cdots  \ar[r]&  H^{i-2}(B, d_{\omega}) \ar[r]& H^{i}(B, d_{\omega}) \ar[r]&
H^{i}(A, d_{\omega}) \ar[r]& H^{i-1}(B, d_{\omega}) \ar[r]&  \cdots 
}.
\end{equation}

For Claim \eqref{r2}, the relevant portion of the long exact sequence is
\begin{equation}
\label{eq:les2}
\xymatrix{ 0\ar[r]& H^{1}(B, d_{\omega}) \ar[r]&
H^{1}(A, d_{\omega}) \ar[r]& H^{0}(B, d_{\omega}) \ar[r]&  \cdots 
}.
\end{equation}
If $\omega \ne 0$, then $H^{0}(B, d_{\omega})=0$ by connectedness of $B$, 
and therefore $\varphi$ induces an isomorphism between $H^{1}(B, d_{\omega})$ 
and $H^{1}(A, d_{\omega})$.  The same thing happens when $\omega =0$, 
since $[e]\ne 0$ by assumption. The claim follows.
\end{proof}

We may take in the above proposition $B= (H(\Sigma_g), d=0)$, where $\Sigma_g$ is
a compact Riemann surface of genus $g>0$, and $e\in H^2(\Sigma_g)$ equal to the
orientation class. 

\begin{corollary}
\label{cor:rk1res}
Let $A=(H^{\hdot}(\Sigma_g)\otimes_e \bwedge (t), d)$ be the 
corresponding Hirsch extension. Then $\RR^1_1 (A)= \{ 0\}$ if $g=1$, 
and $\RR^1_1 (A)= H^{1}(\Sigma_g)$ if $g>1$.
\end{corollary}

\begin{proof}
It is immediate to check that $\RR^1_1 (B)= \{ 0\}$ for $g=1$ 
and $\RR^1_1 (B)= H^{1}(\Sigma_g)$ for $g>1$.  
\end{proof}    

\subsection{Orbifold fundamental groups}
\label{subsec:orbifold}

Imposing weaker assumptions in Theorem \ref{thm:nformal}, we can still derive 
interesting topological consequences. Let $f\colon \pi_1 \surj \pi_2$ be an epimorphism 
between finitely generated groups, and let $\iota \colon G\to \GL(V)$ be a rational 
representation of a linear algebraic group. We then have a natural closed algebraic embedding,
$f^{!} \colon \Hom (\pi_2, G) \inj \Hom (\pi_1, G)$. An application of the Hochschild--Serre 
spectral sequence (see e.g.~\cite[\S VII.6]{Br}) shows that the morphism 
$f^{!}$ preserves characteristic varieties in degree $1$ and induces closed 
algebraic embeddings, $\VV^1_s (\pi_2, \iota) \inj \VV^1_s (\pi_1, \iota)$, 
for all values of $s$. 

When $M$ is an almost free $K$-manifold, we recall from \cite[Theorem 4.3.18]{BG} 
that the projection $p\colon M\to M/K$ induces a natural epimorphism 
$f\colon \pi_1 (M) \surj \pi_1^{\orb} (M/K)$ between orbifold fundamental groups.

\begin{theorem}
\label{thm:monotr}
Suppose that the $K$-action on $M$ is almost free and  the transgression 
$P^{\hdot} \to H^{\hdot +1}(M_K) \equiv H^{\hdot +1}(M/K)$ is injective in degree $1$.
Then the following hold.
\begin{enumerate}
\item \label{mn1}
If the orbit space $N=M/K$ has a $2$-finite $2$-model over $\k \subseteq \C$,
then the homomorphism $f\colon \pi_1 (M) \surj \pi_1^{\orb} (N)$ 
induces an analytic isomorphism 
between $\VV^1_s (\pi_1^{\orb} (N))_{(1)}$ and $\VV^1_s (\pi_1 (M))_{(1)}$, 
for all $s$.
\item \label{mn2}
If $N$ is $2$-formal, then $f$ induces an analytic isomorphism between the 
germs at $1$ of $\Hom (\pi_1^{\orb} (N), G)$ and $\Hom (\pi_1 (M), G)$, for 
$G=\SL_2 (\C)$ or a Borel subgroup.
\end{enumerate}
\end{theorem}

\begin{proof}
Let $B$ be a $2$-finite $2$-model of $N$. By Theorem \ref{thm:rs}, 
$\Omega (M) \simeq_1 \Omega (N) \otimes_{\sigma} \bwedge P^1$.
By Lemma \ref{lem=qhirsch}, the Hirsch extension 
$A=B \otimes_{\tau} \bwedge P^1$
is a $1$-finite $1$-model for $\pi_1(M)$. Using \cite[pp.~117--118]{BG}, we deduce that
$\pi_1^{\orb} (N)$ and $\pi_1 (N)$ share the same $1$-minimal model. Therefore, 
$B$ is a $1$-finite $1$-model for both $\pi_1 (N)$ and $\pi_1^{\orb} (N)$. 
We infer from Theorem \ref{thm:germs} that 
\[
\VV^1_s (\pi_1(M))_{(1)}\cong \RR^1_s (A)_{(0)} \quad \text{and}\quad 
 \VV^1_s (\pi_1^{\orb} (N))_{(1)} \cong \RR^1_s (B)_{(0)}.
\]

Now choose a basis $\{ t_1,\dots,t_r \}$ of $P^1$, and set $e_{\alpha}= \tau( t_{\alpha})$. 
Our hypothesis on the transgression implies that the classes $[e_1],\dots, [e_r]$ are linearly 
independent in $H^2(B)$. Let $\varphi \colon B\inj A$ be the canonical $\cdga$ 
inclusion. Setting $A_{\alpha}=B \otimes_{\tau} \bwedge (t_1,\dots ,t_{\alpha})$, 
we infer from Proposition \ref{prop:circleres}\eqref{r2} that the map 
$\varphi^*\colon H^1(B) \to H^1(A)$ is an isomorphism, and that 
$\RR^1_s (A_{\alpha})= \RR^1_s (A_{\alpha +1})$. Hence, 
$\RR^1_s (B)=\RR^1_s (A)$, for all $s$. Consequently, the germs 
$\VV^1_s (\pi_1(M))_{(1)}$ and $\VV^1_s (\pi_1^{\orb} (N))_{(1)}$ 
have the same analytic local ring $R_M \cong R_N$ (up to isomorphism).

Moreover, the morphism $f^{!} \colon \VV^1_s (\pi_1^{\orb} (N))_{(1)} \inj \VV^1_s (\pi_1(M))_{(1)}$ 
induces a surjection, $f^{!} \colon R_M \surj R_N$ between the corresponding analytic 
local rings. By the well-known Hopfian property for Noetherian rings (see for instance 
\cite[p.~65]{T}), the map $f^{!}$ must be an isomorphism, which completes the proof 
of claim \eqref{mn1}.

To prove claim \eqref{mn2}, 
we use a similar argument, with the $\dga$ $H^{\hdot}:= (H^{\hdot}(N), d=0)$ 
replacing $B$. By Theorem \ref{thm:germs}, there exist isomorphisms 
$\Hom (\pi_1 (M), G)_{(1)} \cong \F (A, \g)_{(0)}$ and 
$\Hom (\pi_1^{\orb} (N), G)_{(1)} \cong \F (H, \g)_{(0)}$. 
By Proposition \ref{prop:flataeq}, $\F(A,\g)$ equals $\F^1(A,\g) \cup \varphi^{!} (\F(H,\g))$.
Since $\varphi^*$ is an isomorphism in degree one, $\F^1(A,\g) = \varphi^{!} (\F^1(H,\g))$.
Therefore, $\F(A,\g) = \varphi^{!} (\F(H,\g)) \cong \F(H,\g)$. We may then take germs at $0$ 
and conclude as before, by a Hopfian argument.
\end{proof}

\begin{example}
\label{ex=tauinj}
As usual, let $K$ be a compact, connected Lie group, and identify 
$H^{\hdot}(K,\Q)$ with $\bigwedge P^{\hdot}_K$. Let $N$ be a 
compact, formal manifold, and assume $b_2(N)\ge s$, 
where $s=\dim P^1_K$  (for instance, take $N$ to be the 
product of at least $s$ compact K\"{a}hler manifolds). 
There is then a degree-preserving 
linear map, $\tau \colon P^{\hdot}_K \to H^{\hdot +1}(N,\Q)$, which is  
injective in degree $1$.  By Theorem \ref{thm:krealiz}, such a map 
can be realized as the transgression in a principal $K$-bundle $M_{\tau}\to N$,  
and the manifold $M_{\tau}$ satisfies the assumptions from Theorem \ref{thm:monotr}. 

\end{example}

Theorem \ref{thm:monotr} may also be applied to a Seifert fibered $3$-manifold 
with non-zero Euler class, $p\colon M \to M/S^1 =\Sigma_g$. 
We obtain the following relations between rank $1$ jump loci and rank $2$
representation varieties of $M$ and the corresponding, intensively studied, 
objects for the compact Riemann surface $\Sigma_g$.

\begin{corollary}
\label{cor=seifert}
In the above setup, $\VV^1_s (M)_{(1)}$ is isomorphic to $\VV^1_s (\Sigma_g)_{(1)}$, 
for all $s$, while $\Hom (\pi_1 (M), G)_{(1)}\cong 
\Hom (\pi_1 (\Sigma_g), G)_{(1)}$, for $G=\SL_2 (\C)$ or
a Borel subgroup.
\end{corollary}

\section{Sasakian manifolds}
\label{sect:sasaki}

In this section, we apply the results from the preceding chapter to obtain 
significant consequences for the topology of compact Sasakian manifolds, 
related to formality properties, representation varieties and cohomology jump loci.

\subsection{Sasakian geometry}
\label{subsec:sas}

A general reference for Sasakian geometry is the monograph 
of Boyer and Galicki \cite{BG}.  

Let $M^{2n+1}$ be a compact Sasakian manifold of dimension $2n+1$.  
By work of Ornea and Verbitsky \cite{OV}, we may assume that the Sasakian 
structure is quasi-regular.   A basic structural result in Sasakian geometry 
(Theorem 7.1.3 from \cite{BG}) guarantees that, in this case, $M$ 
supports an almost free circle action.  Furthermore, 
the quotient space, $N=M/S^1$, is a compact K\"{a}hler orbifold, with 
K\"{a}hler class $h\in H^2(N,\Q)$ satisfying the Hard Lefschetz property, 
that is, multiplication by $h^k$ defines an isomorphism 
\begin{equation}
\label{eq:hardlef}
\xymatrix{H^{n-k}(N) \ar^{\cong}[r] & H^{n+k}(N)}
\end{equation}
for each $1\le k\le n$; see 
\cite[Proposition 7.2.2 and Theorem 7.2.9]{BG}.

The thesis of Tievsky \cite[\S 4.3]{Ti} provides a very useful model 
for a Sasakian manifold.

\begin{theorem}[\cite{Ti}]
\label{thm:sasmodel}
Every compact Sasakian manifold $M$ admits as a finite model over $\R$ the Hirsch 
extension 
\[
A^{\hdot}(M)=(H^{\hdot}(N)\otimes _h \bwedge(t), d),
\] 
where 
$d$ is zero on $H^{\hdot}(N)$ and $dt=h$, 
the K\"{a}hler class of $N$.  
\end{theorem}

Sasakian geometry is an odd-dimensional analogue of K\"{a}hler 
geometry.  From this point of view, the above theorem is a 
rough analogue of the main result on the algebraic topology 
of compact K\"ahler manifolds from \cite{DGMS}, guaranteeing 
that such manifolds are formal. Theorem \ref{thm:sasmodel} only says that 
$M$ behaves like an almost free compact $S^1$-manifold with formal orbit space.
A recent result from \cite{BB+} establishes the formality of the orbifold de Rham algebra 
of a compact K\"{a}hler orbifold. Unfortunately, this is not enough for applying Theorem 
\ref{thm:nformal}, since the authors of \cite{BB+} do {\em not} prove that the orbifold
de Rham algebra is weakly equivalent to the Sullivan de Rham algebra. 

By construction, the Tievsky model $A^{\hdot}(M)$ is a real $\dga$ defined 
over $\Q$. Nevertheless, it does {\em not}\/ follow from \cite{Ti} that $A^{\hdot}(M)$ 
is a model for $M$ over $\Q$.  To understand why that is the case, let us 
make a  parenthetical remark.

\begin{remark}
\label{rem:non-descent}
It follows from Sullivan's work \cite{Su77} that there exist smooth 
manifolds that have non-quasi-isomorphic rational models which 
become quasi-isomorphic when tensored with $\R$ (for concrete examples 
of this sort, see \cite{Op}). Such failure of descent from real homotopy 
type to rational homotopy type may even occur with models endowed 
with $0$-differentials.  
\end{remark}

However, we can say something very useful 
regarding rational models for Sasakian manifolds.  We start 
with a lemma, and will come back to this point in 
Theorem \ref{thm:highf-bis}.

\begin{lemma}
\label{lem:qtievsky}
The Tievsky model $A^{\hdot}_{\R}(M)=(H^{\hdot}(N, \R)\otimes _h \bwedge(t), d)$ 
is a finite model with positive weights for $M$, in the sense from \S\ref{subsec:germs}. 
\end{lemma}

\begin{proof}
By Theorem \ref{thm:sasmodel}, we have that $\Omega (M, \R) \simeq A_{\R}$, 
where $A_{\k}=H_{\k} \otimes_h \bwedge (t)$ and
$H_{\k}$ denotes the $\dga$ $(H^{\hdot}(N, \k), d=0)$. 
Note that $\Omega (M, \R)$, $A_{\R}$, and $H_{\R}$ all 
have rational forms, namely, $\Omega (M, \Q)$, $A_{\Q}$, and 
$H_{\Q}$. The fact that $A_{\R}$ has positive weights follows from the discussion
in \S \ref{subsec:germs}. We only need to show that Tievsky's identification,
$H^{1}(A_{\R}) \equiv H^{1}(M, \R)$, preserves $\Q$-structures. 

Let $\varphi \colon H_{\R} \to  A_{\R}$ be the canonical $\dga$ inclusion.
It follows from the construction of the Tievsky model that the homomorphism 
$\varphi^*\colon H^{\hdot} (N, \R) \to  H^{\hdot}(A_{\R}) \equiv H^{\hdot}(M, \R)$ 
coincides with the induced homomorphism 
$p^* \colon  H^{\hdot} (N, \R) \to H^{\hdot}(M, \R)$,
where $p\colon M \to N=M/S^1$ is the canonical projection. 
In particular, the respective $\Q$-structures are preserved. 
Moreover, since $h\ne 0$, the map $\varphi^*\colon H^1(N,\R)\to H^1(A_{\R})$ 
is an isomorphism, and we are done.
\end{proof}

Applying now Theorem \ref{thm:mpps-bis}, 
we obtain the following corollary.  

\begin{corollary}
\label{cor:tmodel}
Let $M$ be a compact Sasakian manifold. 
For each $i, s\ge 0$, all irreducible components 
of the characteristic variety $\VV^i_s(M)$ passing through $1$ 
are algebraic subtori of the character group $H^1(M,\C^*)$.
\end{corollary}

A well-known, direct relationship between K\"{a}hler and Sasakian 
geometry is as follows.  Let $N$ be a compact K\"{a}hler manifold 
such that the K\"ahler class is integral, i.e., $h\in H^2(N,\Z)$, 
and let $M$ be the total space of the principal $S^1$-bundle 
classified by $h$.  Then $M$ is a regular Sasakian 
manifold.  A concrete class of examples is provided by the 
Heisenberg manifolds $\Heis_n$ from Example \ref{ex:heis} below.

\subsection{Partial formality of Sasakian manifolds}
\label{subsec:sas1f}

Let $M^{2n+1}$ be a compact Sasakian manifold, 
with fundamental group $\pi=\pi_1(M)$.  A basic question one 
can ask is:  Is the group $\pi$ (or, equivalently, the manifold $M$) 
$1$-formal?   
When $n=1$, clearly the answer is negative, a simple 
example being provided by the Heisenberg manifold 
$\Heis_1$.  In \cite[Theorem 1.1]{Ka14}, H.~Kasuya claims  that the 
case $n=1$ is exceptional, in the following sense.

\begin{claim}
\label{thm:sas1f}
Every compact Sasakian manifold of dimension $2n+1$ is 
$1$-formal over $\R$, provided $n>1$.
\end{claim}

It turns out that the proof from \cite{Ka14} has a gap, which we now 
proceed to explain. Given a \cdga $A$, the (degree $2$)  decomposable part
is the subspace $DH^2(A)\subseteq H^2(A)$ defined as the image of 
the product map in homology, $H^1(A)\wedge H^1(A)\to H^2(A)$. 
What Kasuya actually shows is that 
\begin{equation}
\label{eq:kdec}
DH^2(\M_1(M)) = H^2(\M_1(M)),
\end{equation}
for a compact Sasakian manifold $M^{2n+1}$ with $n>1$, 
where $\M_1(M)$ is the Sullivan $1$-minimal model of $M$, over $\R$.

Equality \eqref{eq:kdec} is an easy consequence of $1$-formality.  
Kasuya deduces the $1$-formality of $M$ from  \eqref{eq:kdec}, 
by invoking in \cite[Proposition 4.1]{Ka14} as a crucial tool 
Lemma 3.17 from \cite{ABC}.   
Unfortunately, though, this lemma is false, as shown by M\u{a}cinic in 
Example 4.5 and Remark 4.6 from \cite{Ma}.  Nevertheless, the next 
theorem proves Claim \ref{thm:sas1f} in a stronger form, while also 
recovering equality \eqref{eq:kdec}.

\begin{theorem}
\label{thm:highf}
Every compact Sasakian manifold $M$ of dimension $2n+1$ is 
$(n-1)$-formal, over an arbitrary field $\k$ of characteristic $0$. 
\end{theorem}

\begin{proof}
Let $N=M/S^1$.  By Theorem \ref{thm:sasmodel}, the manifold 
$M$ admits the Tievsky model 
$A_{\C}=(H^{\hdot}(N)\otimes _h \bwedge(t), d)$, 
with $d=0$ on $H^{\hdot}(N)$ and $dt=h$. 
Recall now that the K\"{a}hler orbifold $N$ satisfies the 
Hard Lefschetz Theorem, as explained in \eqref{eq:hardlef}.  
It follows that  the sequence $\{ h\}$ is $(n-1)$-regular in 
$H^{\hdot}(N)$, in the sense from \S \ref{subsec:partform-circle}.
Hence, by Theorem \ref{thm:lefred}, the manifold $M$ is 
$(n-1)$-formal over $\C$.  By descent of partial formality 
(cf.~\cite{SW}), $M$ is $(n-1)$-formal over $\Q$, and hence over $\k$.  
\end{proof}

The next result makes Theorem \ref{thm:highf} more precise, 
by constructing an explicit finite, $(n-1)$-model with zero differential 
for $M$ over an {\em arbitrary}\/ field of characteristic $0$. 

\begin{theorem}
\label{thm:highf-bis}
Let $M$ be a compact Sasakian manifold $M$ of dimension $2n+1$. 
The Sullivan model of $M$ over a field $\k$ of 
characteristic $0$ has the same $(n-1)$-type over $\k$ as the \cdga~
$(H^{\hdot}(N)/h\cdot H^{\hdot}(N), d=0)$, 
where $N=M/S^1$ and $h\in H^2(N, \k)$ is the K\"ahler class.
\end{theorem}

\begin{proof}
As before, let $\Omega (M, \k)$ be the Sullivan model of $M$ over $\k$, 
let $A_{\R}$ be Tievsky model of $M$ over $\R$, 
and let $H_{\k}=(H^{\hdot}(N, \k), d=0)$.
Recall from the proof of Theorem \ref{thm:lefred} that there is a $\dga$ 
map $\psi \colon A_{\R} \to (H_{\R}/h H_{\R}$, $ d=0)$ 
which induces a graded ring isomorphism between the truncations 
$H^{\le n}(A_{\R})$ and $(H_{\R}/h H_{\R})^{\le n}$,
with inverse induced by $\varphi^*$. Identify the graded rings 
$H^{\le n}(M, \R)$ and $H^{\le n}(A_{\R})$,
using the zig-zag of quasi-isomorphisms provided by the weak 
equivalence $\Omega (M, \R) \simeq A_{\R}$. The proof of 
Lemma \ref{lem:qtievsky} shows that the composed graded ring isomorphism,  
$H^{\le n}(M, \R) \cong (H_{\R}/h H_{\R})^{\le n}$, respects $\Q$-structures. 
Therefore, we obtain an isomorphism of graded rings,
\begin{equation}
\label{eq=qdescent}
H^{\le n}(M, \Q) \cong (H_{\Q}/h H_{\Q})^{\le n} .
\end{equation} 

The following version of the truncation construction from \S \ref{subsec:res} 
will be helpful in the sequel.  Let $A$ be a connected $\dga$ over a field of 
characteristic zero, and let $q\ge 0$. Write $A^{q+1}= Z^{q+1}(A) \oplus U^{q+1}$.
Plainly, $U^{q+1}\oplus \bigoplus_{j>q+1}A^j$ is a differential ideal. 
Denote by $A[q+1]$ the quotient $\dga$,
and let $\kappa \colon A \to A[q+1]$ be the canonical $\dga$ projection. 
It is immediate to check that the map $\kappa^*\colon H^{\hdot}(A)\to H^{\hdot}(A[q+1])$
is an isomorphism up to degree $q+1$, and that $A^{\hdot}[q+1]=0$ in 
degrees $>q+1$. In particular, $A\simeq_q A[q+1]$
and the truncated cohomology ring $H^{\le q+1}(A)$ is isomorphic 
to $H^{\hdot}(A[q+1])$. 

Let $\M_{\Q}$ be the minimal model of $\Omega (M, \Q)$. As we explained 
previously, $\Omega (M, \Q) \simeq_{n-1} \M_{\Q}[n]$. This clearly implies that 
$(H^{\hdot} (M, \Q), d=0) \simeq_{n-1} (H^{\hdot} (\M_{\Q}[n]), d=0)$. On the other hand,
$\Omega (M, \Q) \simeq_{n-1} (H^{\hdot} (M, \Q), d=0)$, by partial formality over $\Q$. Hence, 
\begin{align}
\label{eq:isos}
\notag
\Omega (M, \Q) &\simeq_{n-1} (H^{\hdot} (\M_{\Q}[n]), d=0) \\[-8pt]
&\cong (H^{\le n}(M, \Q), d=0)  \\[-8pt] \notag
& \cong ((H_{\Q}/h H_{\Q})^{\le n} , d=0),
\end{align} 
where the last two isomorphisms are given by 
modified truncation and \eqref{eq=qdescent}. Plainly, 
$(H_{\Q}/h H_{\Q} , d=0) \simeq_{n-1} ((H_{\Q}/h H_{\Q})^{\le n} , d=0)$, 
by standard truncation. Putting things together, we conclude that 
$\Omega (M, \Q) \simeq_{n-1} (H_{\Q}/h H_{\Q}, d=0)$.
By extension of scalars, the same conclusion holds 
over $\k$. This completes the proof of the theorem.
\end{proof}

As illustrated by the next example, the conclusion of 
Theorem  \ref{thm:highf} is optimal. 

\begin{example}
\label{ex:heis}
Let $E=S^1\times S^1$ be an elliptic complex curve, and let $N=E^{\times n}$ 
be the product of $n$ such curves, with K\"{a}hler form $\omega=\sum_{i=1}^n 
dx_i \wedge dy_i$.   The corresponding Sasakian manifold is the 
$(2n+1)$-dimensional Heisenberg nilmanifold $\Heis_n$. 
Theorem \ref{thm:highf} guarantees that $\Heis_n$ is $(n-1)$-formal. 
As noted in \cite[Remark 5.4]{Ma}, though, the manifold 
$\Heis_n$ is {\em not}\/ $n$-formal. We refer to \cite{Ma} for a 
detailed study of the partial formality properties of this family of manifolds. 
\end{example}

\begin{remark}
\label{rem:nil}
The compact Sasakian manifold $\Heis_n$ from the above example may 
be alternatively described as the quotient of a certain nilpotent Lie group 
$H(1,n)$ by a suitable discrete, cocompact subgroup.
The Tievsky model was used in \cite{CDMY} to show that a compact, 
$(2n+1)$-dimensional nilmanifold admits a Sasakian structure 
if and only if it is the quotient of $H(1,n)$ by a discrete, cocompact 
subgroup $\Gamma$.  
\end{remark}

\subsection{Sasakian groups}
\label{subsec:sasgp}

A group $\pi$ is said to be a {\em Sasakian group}\/ if it 
can be realized as the fundamental group of a compact, 
Sasakian manifold.  A major open problem in the field 
(see e.g. \cite[Chapter 7]{BG} or \cite{Chen}) is:
``Which finitely presented groups are Sasakian?" 

A first, well-known obstruction is that the first Betti number $b_1(\pi)$ 
must be even, see for instance the references listed by Chen in \cite{Chen}.  
Much more subtle obstructions are provided by the following result.  
Fix a field $\k$ of characteristic $0$.

\begin{corollary}
\label{cor:sasobstr}
Let $\pi= \pi_1 (M^{2n+1})$ be a Sasakian group.  Then:
\begin{enumerate}
\item \label{s1}
The Malcev Lie algebra $\m(\pi, \k)$ is the lcs completion of the quotient 
of the free Lie algebra $\L (H_1(\pi , \k))$ by an ideal generated in degrees $2$ and $3$.  
Moreover, this Lie algebra presentation can be explicitly described in
terms of the graded ring $H^{\hdot}(M/S^1, \k)$ and the K\"ahler class 
$h\in H^{2}(M/S^1, \k)$.
\item \label{s2}
The group $\pi$ is filtered-formal. 
\item \label{s3}
For every complex linear algebraic group $G$,
the germ at the origin of the representation variety $\Hom(\pi, G)$ 
is defined by quadrics and cubics only.
\end{enumerate}
\end{corollary}

\begin{proof}
If $n>1$, Theorems \ref{thm:highf} and \ref{thm:malholo} imply that  $\m(\pi, \k)$ 
is isomorphic to the lcs completion of the holonomy Lie algebra of $(H/h H, d=0)$, 
where $H^{\hdot}:=H^{\hdot}(M/S^1, \k)$.
Pick $\k$-bases for $H^1$ and $H^2$, say, $\{ h_i \}$ and $\{ h, f_{\gamma} \}$. 
Set $v_{\gamma}= \cup_H^* f_{\gamma}^* \in \L^2 (h_i^*)$. It follows from 
the definitions that the above holonomy Lie algebra has quadratic presentation
\begin{equation}
\label{eq:qusas}
\h ((H/h H, d=0)) = \L (h_i^*) / \, \text{ideal}\, (v_{\gamma}) .
\end{equation}

If $n=1$, it is well-known that the compact K\"{a}hler orbifold $M/S^1$ is a 
genus $g$ smooth projective curve $\Sigma_g$
(see e.g.~\cite[Proposition 4.4.4]{BG}). By \cite{DGMS}, the manifold $\Sigma_g$ is formal. 
We infer from Lemma \ref{lem:newfinite} that $\Omega (M, \k) \simeq A$, 
where $A:=H \otimes_e \bwedge (t)$,
for some $e\in H^2(\Sigma_g, \k)$. Note that $M$ and $M/S^1$ 
have the same first Betti number, since $h\ne 0$.
This implies that $e\ne 0$, since otherwise $b_1(M)=b_1(M/S^1) +1$. 
Normalizing if necessary, we may
thus assume that $e$ is the orientation class of $\Sigma_g$. 

By Theorem  \ref{thm:malholo}, the Malcev Lie algebra $\m(\pi, \k)$ is isomorphic 
to the lcs completion of $\h (A)$, which in turn can be computed 
as in Theorem \ref{thm:filtf}. 
Plainly, $\h (A)=0$ if $g=0$. If $g>0$, let $\{ a_1, b_1, \dots, a_g, b_g \}$ 
be a dual symplectic basis for $H_1$.
Set $u= \sum_i [a_i, b_i]$. We obtain the following cubic presentation:
\begin{equation}
\label{eq:cubsas}
\h (A) \cong \L (a_1, b_1, \dots, a_g, b_g )/ \, \text{ideal}\, ([a_i, u], [b_i, u]) .
\end{equation}

As shown in \cite{SW},  the same result as in the case $n=1$ 
holds for all orientable Seifert fibered $3$-manifolds.
The claim on representation varieties is a consequence 
of Corollary \ref{cor:sophus}.
\end{proof}

As an application of Corollary \ref{cor:tmodel}, we obtain 
another (independent) obstruction to Sasakianity. 

\begin{corollary}
\label{cor:sasgpobs}
Let $\pi$ be a Sasakian group. 
For each $s\ge 0$, all irreducible components 
of the characteristic variety $\VV^1_s(\pi)$ passing through $1$ 
are algebraic subtori of the character group $\Hom(\pi,\C^*)$.
\end{corollary}

By Theorem \ref{thm:sasmodel}, the $\R$-homotopy type of a 
compact Sasakian manifold $M$ depends only on the cohomology ring 
$H^{\hdot}(M/S^1, \R)$ and the K\"{a}hler class $h\in H^2(M/S^1,\Q)$. 
Surprisingly enough, it turns out that the germs at $1$ of certain representation 
varieties and jump loci of $\pi_1(M)$ depend only on the graded cohomology 
ring of $M/S^1$. 

\begin{corollary}
\label{cor=nokclass}
Let $M$ be a compact Sasakian manifold, and let $G$ be either $\SL_2(\C)$, 
or a Borel subgroup.  
Then the germ at $1$ of $\Hom_{\gr}(\pi_1(M), G)$ depends only on the 
graded ring $H^{\hdot}(M/S^1, \C)$ and the Lie algebra of $G$, 
in an explicit way. Similarly, the germs $\VV^1_s(\pi_1(M))_{(1)}$ 
depend (explicitly) only on the graded ring $H^{\hdot}(M/S^1, \C)$, for all $s$. 
\end{corollary}

\begin{proof}
Set $H^{\hdot}=H^{\hdot}(M/S^1, \C)$, and let $A= H \otimes_h \bwedge (t)$ 
be the complex Tievsky model of $M$. Denote by $\varphi \colon (H,d=0) \inj A$ 
the canonical $\dga$ inclusion. 
Let $\g \subseteq \sl_2 (\C)$ be the Lie algebra of $G$. 
Theorem \ref{thm:germs} and Proposition \ref{prop:flataeq} give the following 
isomorphism of analytic germs:
\begin{equation}
\label{eq=sasexpl1}
\Hom_{\gr}(\pi_1(M), G)_{(1)} \cong (\F^1(A,\g) \cup \varphi^{!} (\F(H,\g)))_{(0)} .
\end{equation}

Since $h\ne 0$, the map $\varphi^*\colon H^1\to H^1(A)$ is an isomorphism. 
Hence, $\F^1(A,\g) = \varphi^{!} (\F^1(H,\g))\subseteq \varphi^{!} (\F(H,\g))$. 
Therefore, $\Hom_{\gr}(\pi_1(M), G)_{(1)} \cong \F(H,\g)_{(0)}$. This proves the first claim.

Again by Theorem \ref{thm:germs}, $\VV^1_s(\pi_1(M))_{(1)} \cong \RR^1_s(A)_{(0)}$, 
for all $s$.  We infer from Proposition \ref{prop:circleres}\eqref{r2} that  
\begin{equation}
\label{eq=sasexpl2}
\VV^1_s(\pi_1(M))_{(1)} \cong \RR^1_s((H, d=0))_{(0)}, \, \text{for all} \,  s \,,
\end{equation}
and this proves the second claim.
\end{proof}

\section{Poincar\'{e} duality and cohomology jump loci}
\label{sect:pd-cjl}

In this section, we prove that Poincar\'{e} duality at the 
level of cochains implies twisted Poincar\'{e} duality. 
We illustrate this phenomenon with examples 
coming from almost free $K$-actions.

Let $A$ be a finite-dimensional, 
commutative graded algebra over a characteristic zero field $\k$. 
We say that $A$ is a {\em Poincar\'{e} duality algebra}\/ of dimension $n$ 
(for short, an $n$-\textsc{pda}) if $A^i=0$ for $i>n$ and $A^n=\k$, 
while the bilinear form
\begin{equation}
\label{eq:bilinear}
\xymatrixcolsep{16pt}
\xymatrix{A^{i}\otimes A^{n-i} \ar[r]& A^n=\k}
\end{equation}
given by the product is non-degenerate, for all $0\le i\le n$ 
(in particular, $A$ is connected).  If $M$ is a closed, 
connected, orientable, $n$-dimensional manifold, then, 
by Poincar\'{e} duality, the cohomology algebra $A=H^{\hdot}(M,\k)$ 
is an $n$-\textsc{pda}. 

Now let $A=(A^{\hdot}, d)$ be a \textsc{cdga}.  We say that $A$ is a 
{\em Poincar\'{e} duality differential graded algebra}\/ of dimension $n$ 
(for short, an $n$-\textsc{pd-cdga}) if the underlying algebra $A$ is an 
$n$-\textsc{pda}, and, moreover,  $H^n(A)=\k$, or, equivalently, 
$d A^{n-1}=0$. 

Clearly, if $A$ is an $n$-\textsc{pda}, then $(A, d=0)$ is an 
$n$-\textsc{pd-cdga}.  Hasegawa showed in \cite{Ha} that 
the minimal model for the classifying space of a finitely-generated 
nilpotent group $\pi$ is a \textsc{pd-cdga}. Let $A^{\hdot}$ be a \textsc{cdga}
with $H^1(A)=0$ for which $H^{\hdot}(A)$ is an $n$-\textsc{pda}.
Lambrechts and Stanley showed in \cite{LS} that $A$ is 
weakly equivalent to an $n$-\textsc{pd-cdga}.
The next result is probably known to the experts. 
For the reader's convenience, we include a proof.

\begin{lemma}
\label{lem:circlepd}
Let $A^{\hdot}= B^{\hdot} \otimes_{\tau} \bwedge (t_i)$ be a Hirsch extension with 
variables $t_i$ of degree $m_i$. If $(B^{\hdot}, d_B)$ is an $n$-\textsc{pd-cdga}, then
$(A^{\hdot}, d_A)$ is an $m$-\textsc{pd-cdga}, where $m=n+ \sum m_i$.
\end{lemma}

\begin{proof}
Clearly $A^{\hdot}$ is an $m$-\textsc{pda}. It remains to check that
$d_A(b\otimes t_{i_1} \wedge \cdots \wedge  t_{i_r})=0$, for all $b\in B^q$ such that
$q+\sum_j m_{i_j}=m-1$. Note that the condition on degrees forces $q\ge n-1$.
Indeed, 
\begin{equation}
d_A(b\otimes t_{i_1} \wedge \cdots \wedge  t_{i_r})= d_B( b) \otimes t_{i_1} \wedge \cdots \wedge  t_{i_r}
+ \sum_j \pm b\cdot \tau (t_{i_j} )\otimes  t_{i_1} \wedge \cdots \widehat{t_{i_j}} \cdots \wedge  t_{i_r}.
\end{equation}

In the above, all elements of the form $b\cdot \tau (t_{i_j})$ belong to $B^{>n}$. 
These elements must be equal to zero, since $B$ is an $n$-\textsc{pda}.
By the same argument, $d_B (b)=0$ if $q>n-1$. Finally, if $q=n-1$ then 
again $d_B (b)=0$, by our $n$-\textsc{pd-cdga} assumption on $B$. 
\end{proof}

\begin{corollary}
\label{cor=kpd}
Let $M$ be an almost free $K$-manifold. If $B$ is a finite model 
of the orbit space $N=M/K$ and an $n$-\textsc{pd-cdga}, then 
the finite model of $M$ from Lemma \ref{lem:newfinite}, 
$A=B \otimes_{\tau} \bwedge P$, is an  $(n+ \dim K)$-\textsc{pd-cdga}. 
If $N$ is a formal, closed, orientable, $n$-manifold, we
may take $B^{\hdot}=(H^{\hdot}(N), d=0)$.
\end{corollary}

In the case of principal $K$-bundles, we obtain a more precise result.

\begin{corollary}
\label{cor:newpd}
Let $N$ be a finite space having an $n$-\textsc{pd-cdga} finite model $B$, over $\Q$. 
Let $K$ be an arbitrary compact connected Lie group. Any Hirsch extension, 
$A= B \otimes_{\tau} \bwedge P_K$, may be realized as an  $(n+ \dim K)$-\textsc{pd-cdga}
finite model of a principal $K$-bundle $M_{\tau}$ over $N$. When $N$ is a formal,
closed, orientable, $n$-manifold, we may take $B^{\hdot}=(H^{\hdot}(N), d=0)$.
\end{corollary}

\begin{proof}
The existence of the principal $K$-bundle $M_{\tau}$ with prescribed finite model 
$A= B \otimes_{\tau} \bwedge P_K$ follows from Theorem \ref{thm:krealiz}. In turn, 
Lemma \ref{lem:circlepd} yields the claimed \textsc{pd-cdga} property.
\end{proof}

Let $(A,d)$ be a finite \cdga. For a finite-dimensional vector space $V$,
define an isomorphism 
$\sigma\colon A^1\otimes \gl(V) \isom  A^1\otimes \gl(V^*)$ 
by $\sigma(\eta\otimes g)= - \eta\otimes g^{*}$ 
for $\eta\in A^1$ and $g\in \gl(V)$.  
Identifying $V$ with $\k^m$, this isomorphism 
coincides with the involution $-\id_{A^1}\otimes T$, 
where $T \colon \gl_m (\k) \to \gl_m(\k)$ is matrix transposition.
It is straightforward to verify 
that $\sigma$ induces an isomorphism between the 
corresponding varieties of flat connections, 
\begin{equation}
\label{eq:sigma}
\xymatrix{\sigma\colon \F(A, \gl(V)) \ar^(.52){\simeq}[r]& \F(A, \gl(V^*))}.
\end{equation}

In the next result, covariant derivatives are taken with respect to the identity
representations of $\gl(V)$ and $\gl(V^*)$.

\begin{lemma}
\label{lem:uptosign}
Let $(A^{\hdot} , d)$ be an $n$-\textsc{pd-cdga}, and 
let $\omega\in \F(A,\gl(V))$  be a $\gl(V)$-valued flat connection.  Then, for all $0\le i\le n$, 
the diagram
\[
\xymatrixcolsep{32pt}
\xymatrix{
(A^i)^* \otimes V^* & (A^{i+1})^* \otimes V^* \, \phantom{,}
\ar[l]_{d_{\omega}^*}\\
A^{n-i} \otimes V^* \ar^{\PD}_{\simeq}[u] & A^{n-i-1} \otimes V^* \, , 
\ar[l]_{d_{\sigma(\omega)}} \ar^{\PD}_{\simeq}[u]
}
\]
with vertical arrows induced by Poincar\'e duality isomorphisms, 
commutes up to a $(-1)^{n-i}$ sign.
\end{lemma}

\begin{proof}
Write $\omega = \sum_{\alpha} \eta_{\alpha} \otimes g_{\alpha} \in A^1\otimes \gl(V)$.
Then $\sigma(\omega) = \sum_{\alpha} -\eta_{\alpha} \otimes g^*_{\alpha}$. 
Pick $a\otimes v^* \in A^{n-i-1} \otimes V^*$ and 
$b\otimes u \in A^{i} \otimes V$.  Denoting by 
$\langle\text{- , -}\rangle$ the evaluation maps 
and using formula \eqref{eq:adv} for the covariant derivative, 
we find that
\begin{equation}
\label{eq:c1}
\langle\PD\circ d _{\sigma(\omega)}(a\otimes v^*), b\otimes u \rangle =
da \cdot b\, \langle v^*, u\, \rangle  -\sum_{\alpha} \eta_{\alpha}\cdot a \cdot b 
\, \langle v^*, g_{\alpha} u \rangle
\end{equation}
and 
\begin{align}
\label{eq:c2}
\langle d^* _{\omega}\circ \PD (a\otimes v^*), b\otimes u \rangle &=
\langle \PD (a\otimes v^*) , d_{\omega} (b\otimes u) \rangle  \\
& = a \cdot db\, \langle v^*, u\, \rangle + \sum_{\alpha} a\cdot \eta_{\alpha} \cdot b 
\, \langle v^*, g_{\alpha} u \rangle. \notag
\end{align}

In view of our $n$-\textsc{pd-cdga} assumption on $A$, 
and the fact that $0=d(a\cdot b)=da\cdot b+(-1)^{n-i-1} a\cdot db$, 
the first terms from \eqref{eq:c1} and \eqref{eq:c2} differ by a 
factor of $(-1)^{n-i}$.   Moreover, by graded-commutativity of $A$, 
we also have that $-\eta_{\alpha} a b = (-1)^{n-i} a\eta_{\alpha}b$, 
for all $\alpha$, and this completes the proof.
\end{proof}

The previous lemma leads to the following twisted Poincar\'e duality result.

\begin{corollary}
\label{cor:twpd}
Let $(A^{\hdot} , d)$ be an $n$-\textsc{pd-cdga}, and 
let $\omega\in \F(A,\gl(V))$.  Then 
\[
H^i(A\otimes V,d_{\omega})^* \cong H^{n-i}(A\otimes V^*,d_{\sigma(\omega)}), \, \forall \, i .
\] 
\end{corollary}

This in turn gives rise to Poincar\'e duality for embedded resonance varieties.

\begin{lemma}
\label{lem:mpd}
Let $A^{\hdot}$ be an $n$-\textsc{pd-cdga}. Let $\g$ be either $\gl_m (\k)$ 
or $\sl_m (\k)$.  Denote by $\theta$ either $\id \colon \gl_m(\k)\to \gl_m(\k)$ 
or the inclusion $ \sl_m(\k)\inj \gl_m(\k)$. Then the involution 
$\sigma \colon A^1\otimes \gl_m(\k) \isom  A^1\otimes \gl_m(\k)$ induces 
an algebraic isomorphism of embedded varieties,
\[
\xymatrix{ \sigma\colon (\F(A,\g), \RR^i_s(A,\theta))
\ar^(.5){\cong}[r]&
(\F(A,\g), \RR^{n-i}_s(A,\theta))}, \, \forall \, i,s \, .
\]
\end{lemma}

\begin{proof}
The first case is a direct consequence of Corollary \ref{cor:twpd}.

In the second case, the equality $\F(A,\sl_m(\k))= \F(A,\gl_m(\k)) \cap A^1\otimes \sl_m(\k)$
is an instance of the well-known general formula describing the behavior of 
flat connections with respect to Lie subalgebras; see \eqref{eq:flat coords}. 
In view of Remark 2.5 from \cite{MPPS}, the resonance 
variety $\RR^i_s(A,\theta)$ is the intersection of 
$\RR^i_s(A,\id_{\gl_m(\k)})$ with $A^1\otimes \sl_m(\k)$. 
On the other hand, the involution $\sigma$ leaves invariant the subspace 
$A^1\otimes \sl_m(\k)\subseteq A^1\otimes \gl_m(\k)$.  The 
desired conclusion follows at once.
\end{proof}

We deduce the following topological consequence.  
Let $G$ be either $\GL_m(\C)$ or $\SL_m(\C)$. Denote by $\iota$ either
the identity $\GL_m(\C)\to \GL_m(\C)$ or the inclusion $\SL_m(\C)\inj \GL_m(\C)$. 

\begin{theorem}
\label{thm:invjump}
Let $X$ be a finite space  admitting a finite model $A$ which is an 
$n$-\textsc{pd-cdga} over $\C$. There is then 
an analytic involution of $\Hom (\pi_1(X), G)_{(1)}$ which 
identifies $\VV^{i}_s(X, \iota)_{(1)}$ with $\VV^{n-i}_s(X, \iota)_{(1)}$, for all $i,s$.
Furthermore, in the rank $1$ case, this involution is induced by 
the involution $\rho \mapsto \rho^{-1}$ of the character group $H^1(X, \C^*)$. 
\end{theorem}

\begin{proof}
The general case follows from Theorem \ref{thm:germs} and Lemma \ref{lem:mpd}.
In the rank one case, Theorem B(2) from \cite{DP-ccm} guarantees that the 
identification between $\F (A)_{(0)}\equiv H^1(X, \C)_{(0)}$ and $H^1(X, \C^*)_{(1)}$ 
is given by the exponential map.  On the other hand, in this case $\sigma =-\id$.
\end{proof}

\begin{remark}
\label{rem:mtduality}
Given a finitely generated group $\pi$, the correspondence $\rho \mapsto (\rho^*)^{-1}$
defines an algebraic involution of the representation variety,
$\alpha\colon \Hom_{\gr} (\pi, \GL_m(\C)) \isom \Hom_{\gr} (\pi, \GL_m(\C))$. 
If $M$ is an $n$-dimensional closed, orientable manifold with $\pi=\pi_1(M)$, 
well-known results about Poincar\'e duality with local coefficients (see for instance 
\cite[\S 2]{W}) imply that the global involution $\alpha$ identifies $\VV^{i}_s(M, \iota)$ 
with $\VV^{n-i}_s(M, \iota)$, for all $i,s$,  where $\iota$ is the identity map of 
$\GL_m(\C)$. Theorem \ref{thm:invjump}, then, can be viewed as a local 
analogue of this classical result.
\end{remark}

\section{Quasi-projective manifolds}
\label{sect:qprof}

Another class of examples where our techniques developed so far give strong 
topological consequences is provided by certain complex quasi-projective 
manifolds, closely related to classical $3$-manifold theory. In this section 
we establish the general setup.

\subsection{Admissible maps and rank $1$ jump loci}
\label{subsec:pencils-rk1}

Let $M$ be a {\em quasi-projective manifold}, i.e., an irreducible, 
smooth, complex quasi-projective variety,  and let $S$ 
be a smooth complex curve, i.e., a $1$-dimensional quasi-projective 
manifold.  A regular, surjective map $f\colon 
M\to S$ is said to be an {\em admissible map}\/ if the generic 
fiber of $f$ is connected. The curve $S$ is said to be of general 
type if $\chi(S)<0$.  The set  $\mathcal{E}(M)$ of admissible 
maps onto curves of general type (modulo reparametrization 
at the target) is finite.

It is readily seen that $\VV^1_1(S)=H^1(S,\C^*)$, 
for every curve $S$ of general type.  A celebrated theorem 
of Arapura describes the geometry of the characteristic 
variety $\VV^1_1(M)$, largely in terms of pull-backs 
along admissible maps of the character tori of the target 
curves of general type.  

\begin{theorem}[\cite{Ar}]
\label{thm:admrk1}
The correspondence $f\leadsto f^{!}(H^1(S,\C^*))$ gives a bijection between
the set $\mathcal{E}(M)$ and the set of positive-dimensional irreducible  
components of $\VV^1_1(M)$ passing through the identity of the character group 
$H^1(M,\C^*)$.
\end{theorem}

In particular, the non-trivial part of 
$(\Hom(\pi_1(M),\C^*),\VV^1_1(M))_{(1)}$ 
pulls back from curves of general type, via admissible maps.  
The infinitesimal counterpart of this result is a consequence 
of \cite[Theorem C]{DP-ccm}. 

\begin{theorem}[\cite{DP-ccm}]
\label{thm:dpccm}
For a quasi-projective manifold $M$ with finite model $A$ with positive 
weights, the set $\mathcal{E}(M)$ is in bijection with the set of 
positive-dimensional irreducible components of 
$\RR^1_1(A)\subseteq H^1(A) \equiv H^1(M)$ via the 
correspondence $f\leadsto f^{!}(H^1(S, \C))$. 
\end{theorem}

\subsection{Embedded rank $2$ jump loci and admissible maps}
\label{subsec:pencils-rk2}
We now formulate a rank-$2$ analog of Theorem \ref{thm:admrk1}.

\begin{question}
\label{quest:admrk2}
Let $M$ be a quasi-projective manifold.
Given a rational representation $\iota\colon \SL_2(\C) \to \GL(V)$, 
does the non-trivial part of the germ 
\begin{equation}
\label{eq:germ1}
(\Hom_{{\rm gr}} (\pi_1(M),\SL_2(\C)),\VV^1_1(M,\iota))_{(1)}
\end{equation}
pull back from curves of general type, via admissible maps?
\end{question}

We start by giving a precise infinitesimal analog of this question.
To do this, we first need to review a couple of relevant facts 
about compactifications and Gysin models.

Following \cite{Du}, we say that a divisor $D$ in a 
projective manifold $\overline{M}$ is a {\em hypersurface arrangement}\/  
if all irreducible components of $D$ are smooth, and $D$ coincides 
locally with the union of an arrangement of linear hyperplanes. 
When all these hyperplanes are actually coordinate hyperplanes, 
$D$ is a {\em normal crossing}\/ divisor. 

Now let $M$ be a quasi-projective manifold.  There is then 
a {\em convenient compactification} $M\subset \overline{M}$.  
This means that the complement 
$D=\overline{M}\, \setminus\, M$ is a hypersurface arrangement,   
and every element of $\mathcal{E}(M)$ is represented by 
an admissible map $f\colon M\to S$ which is the 
restriction of a regular map $\bar{f}\colon \overline{M} \to 
\overline{S}$, where $\overline{S}=S\cup F$ is the canonical 
compactification of the curve $S$ (obtained by adding a 
finite set of points $F$), such that $\bar{f}^{-1}(F) \subseteq D$.
(If needed, one may also assume that $D$ is a normal crossing divisor.) 

In the case when $D$ is a normal-crossing divisor, 
Morgan constructed in \cite{Mo} a {\em Gysin model}, $A(\overline{M},D)$, 
for the quasi-projective manifold $M$.  In \cite{Du}, Dupont 
extends this construction to the case when $D$ is an 
arbitrary hypersurface arrangement.  As with Morgan's 
original Gysin model, Dupont's model $A(\overline{M},D)$ 
is a finite model, and is functorial in the appropriate sense. 

We may now rephrase Question \ref{quest:admrk2}.
Let $M$ be a quasi-projective manifold.  As explained in 
Remark \ref{rem:betti0}, we may assume $b_1(M)>0$, 
to avoid trivialities. Let $\g$ be the Lie algebra 
$\sl_2(\C)$ or one of its Borel subalgebras, and let 
$\theta\colon \g \to \gl(V)$ be a finite-dimensional 
representation. 

We will need one more definition, extracted from \cite{MPPS}. 
For a $1$-finite \cdga $A^{\hdot}$ and a finite-dimensional Lie
representation $\theta\colon \g \to \gl(V)$, the set 
\begin{equation}
\label{eq:bigpi}
\Pi(A,\theta)=\{\omega =\eta\otimes g \in \F^1(A,\g) \mid 
\det \theta(g)=0\}
\end{equation}
is a Zariski-closed, homogeneous subvariety of $\F^1(A,\g)$. 
Furthermore, if $H^i(A)\ne 0$, then $\Pi(A,\theta)$ is contained in 
$\RR^i_1(A,\theta)$, cf. \cite[Theorem 1.2]{MPPS}. 
Plainly, the variety $\Pi(A,\theta)$ depends only 
on the vector space $H^1(A)$ and the representation $\theta$.

\begin{question}
\label{quest:infrk2}
Is there a convenient compactification $\overline{M} = M \cup D$  
such that the following equalities hold?
\begin{align}
\label{eq:flatbara}
\F(A(\overline{M},D),\g) = \F^1(A(\overline{M},D),\g)\cup 
\bigcup_{f\in \mathcal{E}(M)}  f^{!} (\F(A(\overline{S},F),\g)).
\\
\RR^1_1(A(\overline{M},D),\theta) = \Pi(A(\overline{M},D),\theta)\cup 
\bigcup_{f\in \mathcal{E}(M)}  f^{!} (\F(A(\overline{S},F),\g)).
\label{eq:resbara}
\end{align}
\end{question}

In the above, we view $\F^1(A(\overline{M},D),\g)$ and 
$\Pi(A(\overline{M},D),\theta)$ as the `trivial parts' of the 
respective varieties, since they depend only on $b_1(M)$ 
and $\theta$. In the case when $\g=\C$ and $\theta=\id_\C$, 
we clearly have $\F^1(A(\overline{M},D),\g)=H^1(M)$ 
and $\Pi(A(\overline{M},D),\theta)=\{0\}$. 

In view of a recent result from \cite{PS-noshi}, a positive answer to the 
global Question \ref{quest:infrk2} would imply a positive 
answer to the local Question \ref{quest:admrk2}.
Equalities \eqref{eq:flatbara} and \eqref{eq:resbara} 
are known to hold for several interesting classes of 
quasi-projective manifolds $M$. Let $W_{\bullet}$ denote 
the Deligne weight filtration on $H^{\bullet}(M)$.

\begin{example}
\label{ex:proj}
If $M$ is a projective manifold (in which case $W_1(H^1(M))=H^1(M)$), 
we may take $D=\emptyset$ and $F=\emptyset$ in the above formulas.  
By \cite{Mo}, we have that 
$A(M,\emptyset)=(H^{\bullet}(M),d=0)$, and similarly for $S=\overline{S}$. 
Using now Corollary 7.2 from \cite{MPPS}, we conclude that 
Question \ref{quest:infrk2} has a positive answer in this case.
\end{example}

\begin{example}
\label{ex:w1h10}
As shown in Theorem 4.2 and Proposition 4.1 from \cite{BMPP},  
in the case when $W_1(H^1(M))=0$ there is a convenient compactification 
$\overline{M}=M\cup D$, where $D$ is a normal-crossings divisor, 
such that the equalities from Question \ref{quest:infrk2} hold.  
Furthermore, according to Theorem 1.3 from \cite{BMPP}, 
Question \ref{quest:infrk2} has a positive answer also for partial 
configuration spaces of smooth projective curves.
\end{example}

Question \ref{quest:infrk2} is analyzed in much detail in follow-up 
work, \cite{PS-natjump, PS-noshi}. In particular, in \cite[Theorem 1.2]{PS-noshi}, 
we reinterpret this question in terms of the local structure of the representation 
variety $\Hom(\pi_1(M),G)$ and of the embedded cohomology 
jump loci $\VV^1_1(\pi_1(M), \iota)$.

\section{Isolated surface singularities}
\label{sect:links}

In this section, we will describe another large class of 
quasi-projective manifolds for which Question \ref{quest:infrk2} 
has a positive answer.

\subsection{A Hirsch extension model}
\label{subsec:hemlink}

Let $X$ be a complex affine surface endowed with a good $\C^*$-action and 
having a normal, isolated singularity at $0$.  Let $M$ be its singularity link 
(a closed, oriented, smooth $3$-dimensional manifold).  The 
punctured surface $X^*=X\setminus \{0\}$ is a quasi-projective 
manifold which deform-retracts onto $M$.  
Moreover, the almost free good $\C^*$-action on $X^*$ restricts 
to an $S^1$-action on $M$ with finite isotropy subgroups. 
In particular, $M$ is an orientable Seifert fibered $3$-manifold. 
The orbit space, $M/S^1=X^*/\C^*$, is a smooth projective 
curve $\Sigma_g$, of genus $g=\frac{1}{2} b_1(M)$, and the regular 
canonical projection, $f\colon X^* \to X^*/\C^*$, induces an isomorphism 
on first homology.  See for instance \cite{DPS-mz} and references therein. 

\begin{prop}
\label{prop:lkmodel}
Let $e\in H^2(\Sigma_g)$ be the orientation class of $\Sigma_g$.  
The  quasi-projective manifold $X^*$ has   
a finite model with positive weights of the form 
\[
A^{\hdot}=(H^{\hdot}(\Sigma_g)\otimes_e \bwedge (t), d),
\]
where $d$ vanishes on $H^{\hdot}(\Sigma_g)$ and $dt=e$.  
Moreover, $(A^{\hdot}, d)$ is a $3$-\textsc{pd-cdga}.
\end{prop}

\begin{proof}
The smooth projective curve $\Sigma_g$ is a formal 
space, and thus admits the finite model 
$B^{\hdot} = (H^{\hdot}(\Sigma_g), d=0)$.  
Let $e'\in H^{2}(\Sigma_g)$ be the Euler 
class from Lemma \ref{lem:newfinite}.  
The $e'\ne 0$, since otherwise $b_1(M)=1+2g$.  
Hence, after normalization if necessary, we 
may assume $e'=e$. The assertion on Poincar\'e duality follows from Corollary \ref{cor=kpd}.
\end{proof}

From now on, we will assume that $g>0$ (the reason for this 
restriction on the genus is explained in Remark \ref{rem:betti0}). 

\begin{prop}
\label{prop:em}
For $X^*$ as above, $\mathcal{E}(X^*)=\emptyset$ if $g=1$ and 
$\mathcal{E}(X^*)=\{f\}$ if $g>1$.
\end{prop}

\begin{proof}
The orbit map $f\colon X^*\to X^*/\C^*=\Sigma_g$ is an orbifold bundle 
map, with connected generic fiber $\C^*$; thus, $f$ is an admissible map.  
Our claims follow from Theorem \ref{thm:dpccm}, Proposition \ref{prop:lkmodel}
and Corollary \ref{cor:rk1res}. 
\end{proof}

\begin{remark}
\label{rem:rk1v}
By Proposition \ref{prop:lkmodel}, $A$ is a $3$-\textsc{pd-cdga}. In particular,
$\RR^i_s(A)=\emptyset$ for $i>3$ and $s>0$. For
$i=0$, it follows directly from definitions \eqref{eq:defr} and \eqref{eq:adv} that 
$\RR^0_1(A)= \{0\}$ and $\RR^0_s(A)= \emptyset$ for $s>1$. 
By Proposition \ref{prop:circleres}\eqref{r2},
$\RR^1_s(A)$ equals $\RR^1_s((H^{\hdot}(\Sigma_g), d=0))$, for all $s$. 
Again from the definitions, and using Poincar\'e duality in $H^{\hdot}(\Sigma_g)$,
it is straightforward to check that $\RR^1_s(A)$ equals $H^{1}(\Sigma_g)$ for $s\le 2g-2$,  
while it is $\{ 0\}$ for $2g-1 \le s\le 2g$ and it is empty for $s>2g$. 
By Lemma \ref{lem:mpd},
$\RR^i_s(A)$ is thus computed for all values of $i$ and $s$. By Theorem \ref{thm:germs}, 
this gives a description of $\VV^i_s(M)_{(1)}$ for all $i,s$, which complements the 
computations from \cite[\S 8]{DPS-mz}.
\end{remark}

\subsection{Rank $2$ flat connections}
\label{subsec:rk2flat}

Set $\pi=\pi_1(M)\cong \pi_1(X^*)$ and 
$H^{\hdot}=H^{\hdot}(\Sigma_g)$.  Furthermore, let $A^{\hdot} = 
(H^{\hdot}\otimes_e \bwedge (t), d)$ be the complex finite model for 
$M\cong X^*$ described in Proposition \ref{prop:lkmodel}. 
Let $\iota \colon \SL_2 \inj \GL_2$ be the defining 
representation, and let $\theta\colon \sl_2 \inj \gl_2$ 
be its tangential representation.  By Theorem \ref{thm:germs}, 
there is an analytic isomorphism of germ pairs, 
\begin{equation}
\label{eq:germpairs}
\xymatrix{(\Hom_{\gr}(\pi,\SL_2) , \VV^i_1(X^*,\iota))_{(1)}  
\ar^(.53){\simeq}[r]&
(\F(A,\sl_2), \RR^i_1(A,\theta))_{(0)}},
\end{equation}
for all $i\ge 0$.  Our next goal is to describe  completely these germs of 
embedded rank $2$ jump loci. 

We start with the varieties of flat connections.  We denote by 
$\varphi\colon (H^{\hdot}, d=0)\inj A^{\hdot}$ the \cdga 
inclusion. Note that $H^1(\varphi)\colon H^1 \to H^1(A)$ is an isomorphism,
since $e\ne 0$; see \S\ref{sect:res1sas}.

\begin{corollary}
\label{cor:flatag}
Suppose $\g$ is a Lie subalgebra of $\sl_2$.  If $g=1$, then $\F(A,\g)=\F^1(A,\g)$, 
whereas if $g>1$, then  $\F(A,\g)=\varphi^{!} (\F(H,\g))$. 
\end{corollary}

\begin{proof}
Our assertions follow from Proposition \ref{prop:flataeq}.
In the second case, we use the fact that $\F^1(A,\g)= \varphi^{!} (\F^1(H,\g))$, 
since $H^1(\varphi)$ is an isomorphism.
In the first case, the $\dga$ $(H, d=0)$ is the cochain algebra of a two-dimensional 
abelian Lie algebra. Hence, $\F(H,\g)=\F^1(H,\g)$,
by \cite[Lemma 4.14]{MPPS}.  
\end{proof}

This corollary and Lemma 7.3 from \cite{MPPS}
provide an explicit description of the variety $\F(A,\sl_2)$. 

\subsection{Rank $2$ resonance}
\label{subsec:rk2res}

To complete the picture, we turn to the resonance varieties, 
$\RR^i_1(A,\theta)$.  
We know from Proposition \ref{prop:lkmodel} that $A^{\hdot}$ is 
a $3$-\textsc{pd-cdga}; in particular, it is concentrated in degrees 
$0\le i\le 3$.  Consequently, $\RR^i_1(A,\theta)=\emptyset$ 
for $i>3$.  In the remaining degrees, we use Poincar\'e duality, 
cf. Lemma \ref{lem:mpd}.

Lemma 3.4 from \cite{MPPS} takes care of the first case ($i=0$):  
\begin{equation}
\label{eq:deg0}
\RR^0_1(A,\theta)=\Pi(A,\theta).
\end{equation}

In the last case ($i=1$), we obtain the following explicit description.   

\begin{prop}
\label{prop:resaformula}
With notation as in \S\ref{subsec:rk2flat}, we have
\[
\RR^1_1(A,\theta) = 
\begin{cases}
\Pi(A,\theta) &\text{if $g=1$},\\ 
\varphi^{!} (\F(H, \sl_2)) &\text{if $g>1$}.
\end{cases}
\]
\end{prop}

\begin{proof}
For $g>1$, we apply Proposition 4.1 from \cite{BMPP} to the one-element
family of \cdga maps $\{ \varphi : H\inj A \}$. We have to check that 
$\RR^1_1 (A)= \im H^1(\varphi)$ and that equality holds in \eqref{eq:flataincl}
for $\g= \sl_2$. The first property follows from Corollary \ref{cor:rk1res}, and
the second property is verified in Proposition \ref{prop:flataeq}. We infer that
$\RR^1_1(A,\theta) =\Pi(A,\theta) \cup \varphi^{!} (\F(H, \sl_2))$. Our claim
then follows from Corollary \ref{cor:flatag}.

The genus $1$ formula is a consequence of Corollary 3.8 from \cite{MPPS}, since
in this case $\F(A, \sl_2)=\F^1(A, \sl_2)$, by Corollary \ref{cor:flatag},
and $\RR^1_1 (A)= \{ 0\}$, by Corollary \ref{cor:rk1res}.
\end{proof}

We are now in a position to prove Theorem \ref{thm:intro7} from the 
Introduction, which may be rephrased as follows.

\begin{theorem}
\label{thm:mainsurf}
For a punctured, quasi-homogeneous surface with 
isolated singularity, 
Question \ref{quest:infrk2} has a positive answer.
\end{theorem}

\begin{proof}
For the purpose of this proof, we will denote by $M$ the 
given punctured, quasi-homogeneous surface.
We will show that \eqref{eq:flatbara} and \eqref{eq:resbara} hold, 
for a convenient compactification $\overline{M}$ 
obtained by adding a normal crossings divisor $D$.  By Proposition 4.1 
from \cite{BMPP} and the discussion following it, it is enough 
to check only equality \eqref{eq:flatbara}.

Set $\overline{A}^{\,\bullet}=A^{\bullet}(\overline{M},D)$ and 
$H^{\bullet}=(H^{\bullet}(\Sigma_g),d=0)$, and note that 
$A^{\bullet}(\Sigma_g,\emptyset)=H^{\bullet}$.
We need to consider the two cases appearing in Proposition \ref{prop:em}. 
In the case when $g=1$, formula \eqref{eq:flatbara} reduces to 
\begin{equation}
\label{eq:flatg1}
\F(\overline{A}, \g) =  \F^1(\overline{A}, \g) .
\end{equation}

Since the map $f^*\colon H^1(\Sigma_g)\to H^1(M)$ is an isomorphism, 
$\F^1(\overline{A}^{\,\bullet}, \g)=f^{!}(\F^1(H^{\,\bullet}, \g))$. 
Hence, in genus $g>1$, formula \eqref{eq:flatbara} becomes 
\begin{equation}
\label{eq:flatg2}
\F(\overline{A}, \g) = f^{!}(\F (H^{\bullet}, \g)).
\end{equation}

Next, we claim it is enough to show that the backward inclusions 
from both \eqref{eq:flatg1} and  \eqref{eq:flatg2} become equalities 
around $0$.  This is due to the fact that all varieties in sight have 
the property that all their irreducible components pass through $0$. 
This property in turn is an easy consequence of the fact that all 
the above varieties are endowed with a positive weight $\C^*$-action. 
We recall from \cite[Example 5.3]{DP-ccm} that the Gysin model $\overline{A}$ 
has positive weights.  Moreover, as explained in  \cite[\S 9.17]{DP-ccm}, 
the associated $\C^*$-action on $\overline{A}\otimes \g$ leaves 
$\F(\overline{A}, \g)$ invariant. The varieties $\F^1(\overline{A}, \g)$ 
and $f^{!}(\F (H^{\bullet}, \g))$ are defined by quadratic equations; 
thus, they are also invariant with respect to the standard, weight-$1$ $\C^*$-action. 
Therefore, we have reduced our proof to  the corresponding germs at the origin.

Let $A^{\bullet}=(H^{\bullet} \otimes_e \bwedge(t),d)$ be the
finite model for $M$ from Proposition \ref{prop:lkmodel}. 
We know from Corollary \ref{cor:flatag} that 
$\F(A^{\bullet}, \g) =  \F^1(A^{\bullet}, \g)$ for $g=1$ and 
$\F(A^{\bullet}, \g) =  \varphi^{!}(\F (H^{\bullet}, \g))$ for $g>1$. 
In particular, these equalities hold in a neighborhood of $0$. 

By Theorem \ref{thm:germs}, 
$\F(\overline{A}^{\bullet}, \g)_{(0)} \cong  \F(A^{\bullet}, \g)_{(0)}$. 
Clearly, $\F^1(\overline{A}^{\bullet}, \g)_{(0)} \cong  \F^1(A^{\bullet}, \g)_{(0)}$
and $f^{!}(\F (H^{\bullet}, \g))_{(0)} \cong  \varphi^{!}(\F (H^{\bullet}, \g))_{(0)}$. 
Hence, $\F(\overline{A}^{\bullet}, \g)_{(0)} \cong  \F^1(\overline{A}^{\bullet}, \g)_{(0)}$ 
for $g=1$ and $\F(\overline{A}^{\bullet}, \g)_{(0)} \cong  f^{!}(\F (H^{\bullet}, \g))_{(0)}$ 
for $g>1$. In both cases, the Hopfian argument from the proof of 
Theorem \ref{thm:monotr} shows that equality holds in a neighborhood 
of $0$, in both \eqref{eq:flatg1} and  \eqref{eq:flatg2}.  This completes the proof. 
\end{proof}

\begin{ack}
This work was started at the Centro di Ricerca Matematica 
Ennio De Giorgi in Pisa, Italy, in February 2015. The authors wish to 
thank the organizers of the Intensive Period on Algebraic Topology, 
Geometric and Combinatorial Group Theory for 
providing an inspiring mathematical environment. 
The work was continued while the second author visited the 
Institute of Mathematics of the Romanian Academy in June, 
2015. He thanks IMAR for its support and warm hospitality. 
We acknowledge with thanks useful discussions with Anca M\u{a}cinic
(on partial formality) and Liviu Ornea (on Sasakian geometry).
\end{ack}

\newcommand{\arxiv}[1]
{\texttt{\href{http://arxiv.org/abs/#1}{arxiv:#1}}}
\newcommand{\arx}[1]
{\texttt{\href{http://arxiv.org/abs/#1}{arXiv:}}
\texttt{\href{http://arxiv.org/abs/#1}{#1}}}
\newcommand{\doi}[1]
{\texttt{\href{http://dx.doi.org/#1}{doi:#1}}}
\renewcommand{\MR}[1]
{\href{http://www.ams.org/mathscinet-getitem?mr=#1}{MR#1}}

\end{document}